\documentclass[a4paper, reqno, 11pt]{amsart}

\usepackage{a4wide}
\usepackage{amsmath,bm}
\usepackage{scrextend} 
\allowdisplaybreaks
\usepackage[normalem]{ulem}
\usepackage[dvipsnames]{xcolor}
\usepackage{setspace}
\usepackage{amsbsy}
\usepackage{hyperref}
\usepackage[overload]{empheq}
\usepackage{derivative, mathrsfs}
\usepackage{multirow}

\usepackage{subcaption}
\usepackage{float}

\usepackage[english]{babel}
\usepackage{amssymb}
\usepackage{enumerate}
\usepackage{ifthen}
\usepackage{bbm}
\usepackage{esint}
\usepackage{graphicx}  % standard latex graphics packages
\usepackage{geometry}
\provideboolean{shownotes} 
\setboolean{shownotes}{true}
\usepackage{color}
\usepackage{dsfont}

\usepackage{dsfont}
\usepackage{hyperref}
\usepackage{setspace}
\usepackage{soul}
\usepackage{xcolor}
\usepackage{mathtools}

\newcommand{\lsim}{{\;\raise0.3ex\hbox{$<$\kern-0.75em\raise-1.1ex\hbox{$\sim$}}\;}}
\newcommand{\gsim}{{\;\raise0.3ex\hbox{$>$\kern-0.75em\raise-1.1ex\hbox{$\sim$}}\;}}
\newcommand{\hole}[1]{
	\ifthenelse{\boolean{shownotes}}%
	{\begin{center} \fbox{ \rule {.25cm}{0cm}
				\rule[-.1cm]{0cm}{.4cm} \parbox{.85\textwidth}{\begin{center}
						\texttt{#1}\end{center}} \rule {.25cm}{0cm}}\end{center}}
	{}
}

\newcommand{\R}{{\mathbb R}} 
\newcommand{\x}{{\mathbf x}}
\newcommand{\del}{\partial}

\newcommand{\eps}{\varepsilon}

\newcommand{\tcb}{\textcolor{blue}}

\newtheorem{theorem}{Theorem}[section]
\newtheorem{proposition}[theorem]{Proposition}
\newtheorem{lemma}[theorem]{Lemma}
\newtheorem{corollary}[theorem]{Corollary}
\newtheorem{definition}[theorem]{Definition}
\allowdisplaybreaks
\theoremstyle{remark}
\newtheorem{remark}[theorem]{Remark}

\numberwithin{equation}{section}

\numberwithin{equation}{section}
\hypersetup{
    colorlinks=true,
    linkcolor=blue,
    filecolor=blue,      
    urlcolor=blue,
}

\begin{document}

    \title[Equilibration vs Localization]{Equilibration versus Localization in a  Diffusion-relaxation system }

        \author[A. AlNajjar]{Aseel AlNajjar}
        \address[Aseel AlNajjar]{\newline Computer, Electrical and Mathematical Science and Engineering Division, King Abdullah University of Science and Technology (KAUST), Thuwal 23955-6900, Saudi Arabia}
        \email{aseel.alnajjar@kaust.edu.sa}

        \author[H. Kim]{Hoyoun Kim}
        \address[Hoyoun Kim]{\newline Computer, Electrical and Mathematical Science and Engineering Division, King Abdullah University of Science and Technology (KAUST), Thuwal 23955-6900, Saudi Arabia}
        \email{hoyoun.kim@kaust.edu.sa}

        \author[A. Tzavaras]{Athanasios Tzavaras}
        \address[Athanasios Tzavaras]{\newline Computer, Electrical and Mathematical Science and Engineering Division, King Abdullah University of Science and Technology (KAUST), Thuwal 23955-6900, Saudi Arabia}
        \email{athanasios.tzavaras@kaust.edu.sa}

	\date{\today}
	
	\begin{abstract}
		We consider a diffusion-relaxation system and investigate the conditions on parameters leading to equilibration versus localization. When the diffusion is dominant, solutions converge toward homogeneous equilibria. By contrast,  when the effective diffusion is weak, localization emerges. Such behaviors have been studied for various models through formal asymptotic arguments and linearized stability analysis, but rigorous understanding of the associated  nonlinear phenomena remains limited, particularly in higher dimensions. In the equilibration regime, we establish convergence toward constant equilibria by exploiting an energy dissipation structure and invariant-region estimates. In the localization regime, we study self-similar solutions and transform the problem of their existence into an autonomous dynamical system. The existence of self-similar profiles associated to localizing solutions is reduced to the construction of a heteroclinic orbit for an autonomous dynamical system. Their existence  is obtained through an application of geometric singular perturbation theory. Our analysis  provides a rigorous characterization of the transition from equilibration to localization.
    \end{abstract}

%        \subjclass[2020]{35D99, 35Q35, 76N10, 76R50, 76T30.}
        %35D30  	Weak solutions to PDEs
        %35D99  	None of the above, but in this section
        %35Q35  	PDEs in connection with fluid mechanics
        %76N10  	Existence, uniqueness, and regularity theory for compressible fluids and gas dynamics
        %76R50  	Diffusion
        %76T30  	Three or more component flows
%        \keywords{T}  
	
	\maketitle
	
%	\tableofcontents

%----------------------------------------------------------------------------------------
%----------------------------------------------------------------------------------------
%----------------------------------------------------------------------------------------
%----------------------------------------------------------------------------------------
%----------------------------------------------------------------------------------------

\section{Introduction}
We consider the system of partial differential equations
\begin{equation}
	\tag{P}
	\begin{aligned}
		\partial_t u &= \Delta (\frac{1}{\gamma}u^n),
		\\
		\partial_t \gamma &= -\gamma +u^m.
	\end{aligned}
	\label{pmi}
\end{equation}
consisting of a nonlinear diffusion equation coupled with a relaxation equation. The system is supplemented with initial data,
\begin{equation}
	u(\textbf{x},0) = u_0(\textbf{x})>0, \quad \gamma(\textbf{x},0) = \gamma_0(\textbf{x})>0, \quad( u_0, \gamma_0) \in L^\infty(\Omega)\times L^\infty(\Omega)
\end{equation}
and the problem is set up either in the whole space $\R^d$ or in a bounded domain $\Omega \subset \R^d$ with smooth boundary,
in which case it is coupled with Neumann boundary conditions
\begin{equation}\label{nbc}
	\frac{\partial u}{\partial \nu} = 0 \hspace{0.2cm}\text{on}\hspace{0.1cm}\partial \Omega, \hspace{0.4cm} \frac{\partial \gamma}{\partial \nu} = 0 \hspace{0.2cm}\text{on}\hspace{0.1cm}\partial \Omega.
\end{equation}
Sometimes, when precise computations are needed, the domain $\Omega$ is selected as $[0,1]^d$.

We study the behavior of solutions as the parameters vary in the range $m, n > 0$. As time
proceeds the relaxation equation will drive the behavior towards the equilibrium curve $\gamma = u^m$. 
According to this premiss,  the effective response of the system is captured by the diffusion equation $\del_t u = \Delta u^{n-m}$ 
which is stable for $n> m$, but unstable and conceivably ill-posed in the range $n < m$.
An asymptotic calculation, using the Chapman-Enskog expansion, indicates  that the next order of the expansion offers a stabilizing mechanism 
in the unstable range (see Section~\ref{sec:CE}). Numerical calculations corroborate this scenario and indicate
that solutions equilibrate in the range $n > m$ and that localized structures emerge within the range $m > n$, see Appendix \ref{appA}.
The goal of the work is to quantitatively analyze the aforementioned mechanism.

The system \eqref{pmi} is motivated by both physical and biological applications, of modeling shear band formation in plasticity  and chemotactic bacterial aggregation. 
These settings, very different in nature, share a common feature of emerging localized structures. We posit that a diffusion--relaxation 
structure can provide a quantitative explanation of that behavior where nonlinear diffusion competes with a relaxation mechanism
that controls the diffusivity.  Depending on the parameter range, this interplay leads to equilibrium  when the diffusion dominates;
alternatively, spatial nonuniformities emerge in the diffusivity and lead to the formation of localized structures.

Shear band formation is a localization phenomenon widely observed in plasticity. Of relevance here is the emergence of coherent localized structures
during high-strain rate deformation of metals often preceding material failure. The study of shear bands, initiated in \cite{ZenerHollomon1944}, 
has attracted ample attention in the mechanics literature, see \cite{Viswanathan2020} for a recent survey of the subject.
We refer to \cite{KatsaounisTzavaras2009} for the mathematical modeling of shear bands and an explanation of the role of relaxation
in that context.
Within this framework, localization is understood not merely as linear ill-posedness, but as a nonlinear outcome of competing mechanisms.
 The emergence of coherent localized structures is rigorously demonstrated through the construction of focusing solutions 
 using dynamical systems methods and geometric singular perturbation theory \cite{KatsaounisOlivierTzavaras2017,LeeTzavaras2017,LeeKatsaounisTzavaras2019}.
 In that case the base solutions are time-dependent which adds complexity and makes the analysis cumbersome.
 The model \eqref{pmi} is intended to capture and elucidate the main mechanism of localization,
 in a simpler setting (than the shear band models) involving the nonlinear stability or instability of steady equilibria (rather than dynamically evolving states). Compared to the literature on shear bands,  it has the added complexity of involving several space dimensions.
 
In a different direction, now  from a biological perspective, a localization mechanism arises in chemotaxis and is classically modeled by Keller--Segel models \cite{KellerSegel1970}. This concentration mechanism driven by advection has been extensively analyzed starting from the pioneering work \cite{JagerLuckhaus1992,HerreroVelazquez1997} and has led to a rich mathematical theory of aggregation and blow-up phenomena, {\it e.g.} \cite{CarrilloEtAl2011}. Comprehensive surveys of the Keller--Segel theory and its variants can be found in \cite{Horstmann2003,Hillen2009}.
Chemotaxis is traditionally modeled through an advection-driven mechanism and does not relate in an obvious way to systems such as \eqref{pmi}. 
Recent studies have incorporated nonlinear mobilities, phase-separation effects, and diffusion--relaxation structures. In particular, aggregation has been interpreted as an instability and phase-separation process \cite{ChoiKim2024}, while logarithmic and diffusion--relaxation chemotaxis models provide simplified formulations of the classical Keller--Segel dynamics while still capturing essential aggregation mechanisms \cite{DesvillettesKimTrescasesYoon2019}. In the present work, we indicate that diffusion--relaxation systems may also exhibit localization phenomena through a mechanism distinct from the classical chemotactic drift.

In order  to set the problem precisely, observe that:
\begin{itemize}
	\item The system \eqref{pmi} admits a two-parameter family of trivial solutions $(u,\gamma) = (a,a^m+b e^{-t})$ for $a, b\ge 0$.
	\item If $u_0 \in L^1(\Omega)$, the system conserves the density $u$ 
    \begin{equation}
        \int_\Omega u(\textbf{x},t) d\textbf{x} = \int_\Omega u_0(\textbf{x}) d\textbf{x} \quad \mbox{for any $t\ge 0$}.
        \label{conslaw}
    \end{equation}
	\item In the case that  $\Omega = \R^d$, the system is scaling invariant in space, namely:  if $u(\textbf{x},t)$ and $\gamma(\textbf{x},t)$ form a solution of \eqref{pmi}, 
	then for  $c>0$ 
	\begin{equation} \label{scin}
		\tilde{u}(\textbf{x},t) = c^a \ u(c \textbf{x}, t), \quad \tilde{\gamma}(\textbf{x},t) = c^{b} \ \gamma(c \textbf{x}, t)
	\end{equation}
	is again a solution of \eqref{pmi}, where the constants 
	\begin{equation} \label{param}
		a := a^{m,n} = {2 \over 1+m-n}, \quad b := b^{m,n} = {2m \over 1+m-n} \, ,
	\end{equation}
	are functions of $m$, $n$ and will play a prominent role throughout this work.
\end{itemize}
In the sequel, we study the nonlinear stability/instability of the equilibria $(a, a^m)$, $a > 0$.
The aim is to present a rigorous analysis of the phenomena of equilibration versus localization in several space dimensions. 
The relevance of the transition is captured, in Section~\ref{sec:CE}, by a Chapman-Enskog expansion of the diffusion--relaxation system \eqref{pmi}
in the zero-relaxation time limit.

We start in Section \ref{sec:linear} with a study of linearized stability of the Neumann problem \eqref{pmi}, \eqref{nbc} for the domain 
$\Omega = [0,1]^d$ which facilitates explicit computation. Analysis of the eigenvalues then predicts
\begin{itemize}
\item Hadamard instability in the parameter range $n = 0$, $m > 0$.
\item Linearized asymptotic stability when $n > m \ge 0$.
\item Turing Type II instability when $m > n >0$.
\end{itemize}

In Section \ref{sec:equilibration} we consider the parameter range $n > m \ge 0$ and establish convergence of solutions to the equilibria $(a,a^m)$
using energy methods, invariant regions and techniques of nonlinear analysis.
The diffusion--relaxation system \eqref{pmi} is endowed with an {\em energy dissipation structure},
\begin{equation}
	\frac{d}{dt}\int_{\Omega} F(u,\gamma) d\x + (n+1)\int_{\Omega} |\nabla \frac{u^n}{\gamma}|^2 d\x + \int_{\Omega} \frac{1}{\gamma^2} \Big((u^m)^{\frac{n+1}{m}}-\gamma^{\frac{n+1}{m}}\Big)(u^m-\gamma) d\x = 0 \, ,
	\label{eq:EDidentity}
\end{equation}
the derivation of which is provided in Lemma~\ref{lem:energy}, where 
\begin{equation}
	F(u,\gamma) := \frac{u^{n+1}}{\gamma}+\frac{m}{n-m+1}\gamma^{\frac{n-m+1}{m}} \, .
\end{equation}
 Note that the first equation in \eqref{pmi} when $\gamma$ fixed is a gradient flow in $H^{-1}$, while the second equation \eqref{pmi} when
  $u$ fixed is  a gradient flow in $L^2$. Equation \eqref{eq:EDidentity} captures the combined energy dissipation induced by the coupling, 
  and exhibits two dissipative mechanisms one by the diffusion of the quantity $\sigma = u^n / \gamma$ (that we call stress), and one by the frictional dissipation of the relaxation. The energy density of the Lyapunov functional \( \int F(u,\gamma) d\x\) is nonnegative whenever \(m < n+1\) 
  and in that range admits the minimal value zero. By contrast, in the range \( n+1 < m\)
the energy may be driven to negative infinity and this range is expected to be inherently unstable. 
In the range \(m < n\), the energy is convex and stabilizes the process; this aspect is analyzed in detail in Section~\ref{sec:equilibration}.

The regime \(n < m < n+1\) is the most interesting, where based on the numerical observations in Figure~\ref{fig:EquilvsLocal}, one expects localization to occur.
The functional \(F(u,\gamma)\) remains bounded from below in that range but is no longer convex, and the energy dissipation identity is valid but no longer leads to equilibration. 
To get a further insight in that region we will consider the feasibility of a class of solutions of the form
\begin{equation}\label{intro-ansatz}
u(\rho,t)=e^{a\lambda t}\bar U(\xi),
\qquad
\gamma(\rho,t)=e^{b\lambda t}\bar\Gamma(\xi),
\qquad
\xi=e^{\lambda t}\rho, \; \rho = |x|
\end{equation}
where $a,b$ are the parameters in \eqref{param} and $\lambda>0$. Such functions will exhibit localizing behavior provided
the profiles are appropriately selected, see Definition \ref{def:localization}. It is shown in Section \ref{sec:localization} that
the profiles $(\bar U, \bar \Gamma)$ must satisfy the boundary value problem
\begin{align*}
&\begin{cases}
\frac{2}{1+m-n} \lambda \bar{U} + \lambda \xi \bar{U}' = \big(\frac{1}{\bar{\Gamma}} \ \bar{U}^n \big)'' + {d-1 \over \xi} \big(\frac{1}{\bar{\Gamma}} \ \bar{U}^n \big)',   \\ 
\frac{2m}{1+m-n} \lambda \bar{\Gamma}+\lambda \xi \bar{\Gamma}'  = -\bar{\Gamma} + \bar{U}^m 
\end{cases} \qquad 0 < \xi < \infty
 \tag{$PS$}
\\
&\bar{U}(0) = U_0 \, , \quad \bar{\Gamma}(0) = 1 \, , \quad \bar{U}^\prime (0) = 0, \quad \bar{\Gamma}^\prime (0)=0
\\
&\bar{U} (\xi) \to 0 \, \quad \bar{\Gamma} (\xi) \to 0 \, , \quad \mbox{as $\xi \to \infty$}
\end{align*}
where $U_0 > 0$ is a parameter connected to the growth rate $\lambda$ via
\begin{equation}
		 a \lambda + 1 =  U_0^m
\end{equation} 

In Sections~\ref{sec:reduction}, \ref{sec:asymptotic}, and \ref{sec:gspt} we devise an approach to solve the problem of determining the localization profiles. This is done by desingularizing the problem and transforming the resulting autonomous system to the construction
of a heteroclinic orbit for the more convenient system \eqref{eq:pqr}. The somewhat elaborate transformations are described in Section \ref{sec:reduction}. 
The construction of the heteroclinic orbit for the system \eqref{eq:pqr} is performed by employing  the geometric theory of singular perturbations 
(an idea developed in the context of shear band problems in \cite{LeeTzavaras2017}). This leads to an existence theory for 
a heteroclinic orbit in the range  $n< m < 1$ with $n$ sufficiently small.
The procedure is carried out in Sections~\ref{sec:asymptotic}, and \ref{sec:gspt} for dimension $d=1$. The existence result is stated in Theorem \ref{thm:exist}. When translated to localizing solutions via the transformation \eqref{intro-ansatz}, it leads to precise properties on 
the growth in time and decay in space of the localizing solutions, which are compared to the ones
obtained numerically in Section \ref{sec:numeric}.

%\newpage

%%%%%%%%%%%%%%%%%%%%%%%%%%%%%%%%%%%%%%%%%%%%%%%%%%%%%%%%%%%%%%%%
\section{Chapman-Enskog expansion of the zero relaxation-time limit}\label{sec:CE}
Our goal in this section is to obtain an effective equation for the zero-relaxation time limit of the system  \eqref{pmi}.
Introduce a parameter $\eps>0$ and use a space-time rescaling
 $ t \to \eps \ t$, $ \x \to \sqrt{\eps} \ \x$
 that preserves the diffusive scaling but observes the long-time asymptotic behavior of solutions in an approximating sense. 
We obtain
\begin{equation} \label{eq:timerescaling}
	\begin{cases}
		\partial_t u = \Delta \big( \frac{1}{\gamma} u^n \big), \\
		\eps \partial_t \gamma = -\gamma + u^m.
	\end{cases}
\end{equation}
Then we consider the Chapman-Enskog expansion for this problem. Let
\begin{align*}
	u &= u_0 + \eps u_1 + O(\eps^2), \\
	\gamma &= \gamma_0 + \eps \gamma_1 + O(\eps^2),
\end{align*}
where $u_i$ and $\gamma_i$ satisfy the Neumann boundary condition and have initial values given by
$u_0(\x,0) = u(\x,0)$, $\gamma_0(\x,0) = \gamma(\x,0)$, $u_k(\x,0) = \gamma_k(\x,0) = 0$, for  $k=1,2,\cdots$.
The asymptotics of the system  \eqref{eq:timerescaling} are expressed using the formulas
\begin{align*}
%	\partial_t u &= \partial_t u_0 + \eps \partial_t u_1 + O(\eps^2) \\
%	\eps \partial_t v &= \eps \partial_t v_0 + O(\eps^2) \\
	\Delta \big( \frac{1}{\gamma} u^n \big) &= \Delta \big( \frac{1}{(\gamma_0 + \eps \gamma_1 + \cdots)} (u_0 + \eps u_1 + \cdots)^n \big) \\ &= \Delta \big( \frac{1}{\gamma_0} (1 - \eps \frac{\gamma_1}{\gamma_0} + \cdots) (u_0^n + \eps n u_0^{n-1} u_1 + \cdots) \big) \\ &= \Delta \big( \frac{u_0^n}{\gamma_0} \big) +\eps \Delta \big( n \frac{u_0^{n-1}}{\gamma_0} u_1 - \frac{u_0^n}{\gamma_0^2} \gamma_1  \big) + O(\eps^2) \\
	-\gamma + u^m & = -(\gamma_0 + \eps \gamma_1  + \eps^2 \gamma_2 + \cdots) + (u_0 + \eps u_1 + \eps^2 u_2 + \cdots)^m \\ &= (u_0^m-\gamma_0) + \eps (mu_0^{m-1} u_1 - \gamma_1 ) + O(\eps^2)
\end{align*}
Collecting the terms up to the $O(\eps)$ order, we have 
\begin{align*} 
	O(1) \  \mbox{order terms} \ :& \ \partial_t u_0 = \Delta \big( \frac{u_0^n}{\gamma_0} \big) \\
	& \ u_0^m-\gamma_0 = 0 \\ 
	O(\eps) \  \mbox{order terms} \ :& \ \partial_t u_1 = \Delta \big( n \frac{u_0^{n-1}}{\gamma_0} u_1 - \frac{u_0^n}{\gamma_0^2} \gamma_1  \big) \\
	& \ \partial_t \gamma_0 = mu_0^{m-1} u_1 - \gamma_1
\end{align*}
The order $O(1)$ term gives the system for $(u_0,\gamma_0)$.
\begin{equation} \label{sysu0}
	\begin{cases}
		\partial_t u_0 = \Delta (u_0^{n-m}), \\
		\gamma_0 = u_0^m.
	\end{cases}
\end{equation}
At the order of $O(\eps)$, we obtain the system for $(u_1,\gamma_1)$.
\begin{equation} \label{sysu1}
	\begin{cases}
		\partial_t u_1 = \Delta \big( (n-m) u_0^{n-m-1} u_1 + mu_0^{n-m-1} \Delta (u_0^{n-m})  \big), \\
		\gamma_1 = mu_0^{m-1} u_1 - mu_0^{m-1} \Delta (u_0^{n-m}).
	\end{cases}
\end{equation}
Combining \eqref{sysu0} and \eqref{sysu1}, we obtain an equation for $u$ that approximates the dynamics of \eqref{eq:timerescaling} up to order $O(\eps^2)$
\begin{align*}
	\partial_t u &= \partial_t u_0 + \eps \partial_t u_1 + O(\eps^2) 
	\\ 
	&= \Delta (u_0^{n-m}) + \Delta \big( \eps (n-m) u_0^{n-m-1} u_1 + \eps mu_0^{n-m-1} \Delta (u_0^{n-m})  \big) + O(\eps^2) 
	\\ 
	&= \Delta \big( (u_0+\eps u_1 + O(\eps^2))^{n-m} + \eps mu_0^{n-m-1} \Delta (u_0^{n-m})  \big) + O(\eps^2) \\ &= \Delta \big( u^{n-m} + \eps mu^{n-m-1} \Delta (u^{n-m})  \big) + O(\eps^2)
\end{align*} 
Hence, the effective equation satisfied by $u$ within $O(\eps^2)$ is
\begin{equation} \label{eq:chapmanenskog}
	\partial_t u =\Delta \big( u^{n-m} + \eps mu^{n-m-1} \Delta (u^{n-m})  \big).
\end{equation}
%\tcr{When $m = n+1$, this is the case of the paper [Katsaounts and Tzavaras, 2011].}
%Setting  $F(u) = u^{n-m}$ this reads
%\begin{equation}
%	\partial_t u =\Delta \big( F(u) + \frac{\eps m}{n-m} F'(u) \Delta F(u)  \big).
%\end{equation}
%In turn, multiplying by $F(u)$ and $F'(u)\Delta F(u)$, we derive the following inequalities.
%\begin{lemma} \label{lem:chapman}
%	Denote $G'(u) = F(u)$. Then \eqref{eq:chapmanenskog} satisfies the energy identities
%	\begin{align*}
%		& \partial_t \int_\Omega G(u) d\x + \int |\nabla F(u)|^2 d\x = \frac{\eps m}{n-m}\int  F'(u) |\Delta F(u)|^2 d\x, \\
%		& \frac{1}{2}\partial_t\int |\nabla F(u)|^2 d\x + \int F'(u) |\Delta F(u)|^2 d\x = \frac{\eps m}{n-m} \int |\nabla (F'(u) \Delta F(u) ) |^2 d\x.
%	\end{align*}
%\end{lemma}

\smallskip
In summary, the asymptotics and stability considerations suggest two cases:
\begin{enumerate}
	\item If $n>m$, then the $O(\eps)$ approximation is the porous media equation 
	\begin{equation}
		\partial_t u = \Delta u^{n-m}
	\end{equation}
	which is stable.
%	
%	
%	\begin{align*}
%		& \partial_t \int_\Omega G(u) d\x + \int |\nabla F(u)|^2 d\x = 0, \\
%		& \frac{1}{2}\partial_t\int |\nabla F(u)|^2 d\x + \int F'(u) |\Delta F(u)|^2 d\x = 0.
%	\end{align*}
	
	\item If $n<m$, then the leading order approximation is unstable; in fact, it exhibits Hadamard-instability.
	The $O(\eps^2)$ approximation becomes the Cahn-Hilliard  equation of the form:
	\begin{equation}
		\partial_t u = \Delta \big( u^{n-m} + \eps mu^{n-m-1} \Delta (u^{n-m})  \big),
	\end{equation}
\end{enumerate}	
	
	Consider a perturbation $u := 1+w$ of \eqref{eq:chapmanenskog}  around a constant equilibrium: then 
	$w$ satisfies 
	\begin{equation*}
		\frac{1}{|\Omega|}\int_\Omega u d\x = 1, \qquad \int_\Omega w d\x = 0,
	\end{equation*}
	and the linearized equation reads
	\begin{align}\label{lineqn}
		\partial_t w &= (n-m) \Delta w + \eps m(n-m) \Delta^2 w .
	\end{align}
	%\color{blue}
	On a bounded domain, we impose to \eqref{lineqn} the boundary conditions
	\begin{equation*}
		\nu\cdot\nabla w=0, \qquad \nu\cdot\nabla\Delta w=0.
	\end{equation*}
	One then easily obtains the energy identity
	\begin{equation*}
		\frac{1}{2}\frac{d}{dt}\int_\Omega w^2\,d\x+\eps m(m-n)\int_\Omega|\Delta w|^2\,d\x=(m-n)\int_\Omega|\nabla w|^2\,d\x.
	\end{equation*}
	When $m>n$, the destabilizing effect $(m-n)\int_\Omega|\nabla w|^2\,d\x$ of the backward diffusion term is counteracted by the stabilizing effect  of the fourth-order term. Using eigenmode analysis, in the spirit performed in the following section for the general linearized problem, 
	one can see that the high-frequency terms get stabilized.
	\color{black}

%
%
%\smallskip
%\noindent
%Based on the energy estimate of Lemma~\ref{lem:chapman}, we distinguish two cases:
%\begin{enumerate}
%	\item If $n>m$, then $G(u)$ of Lemma~\ref{lem:chapman} is a convex functional and the following porous media equation is well-posed;
%	\begin{equation}
%		\partial_t u = \Delta F(u),
%	\end{equation}
%	which approximates the main equation \eqref{pmi} up to $O(\eps)$ order. It  satisfies the energy dissipation identities.
%	\begin{align*}
%		& \partial_t \int_\Omega G(u) d\x + \int |\nabla F(u)|^2 d\x = 0, \\
%		& \frac{1}{2}\partial_t\int |\nabla F(u)|^2 d\x + \int F'(u) |\Delta F(u)|^2 d\x = 0.
%	\end{align*}
%	
%	\item If $n<m$, then $G(u)$ of Lemma~\ref{lem:chapman} is a concave functional and a Cahn-Hilliard type equation appears:
%	\begin{equation}
%		\partial_t u = \Delta \big( u^{n-m} + \eps mu^{n-m-1} \Delta (u^{n-m})  \big),
%	\end{equation}
%	which approximates the equation \eqref{pmi} up to the $O(\eps^2)$ order. It satisfies the energy estimates:
%	\begin{align*}
%		& \partial_t \int_\Omega G(u) d\x + \int |\nabla F(u)|^2 d\x = \frac{\eps m}{n-m}\int  F'(u) |\Delta F(u)|^2 d\x, \\
%		& \frac{1}{2}\partial_t\int |\nabla F(u)|^2 d\x + \int F'(u) |\Delta F(u)|^2 d\x = \frac{\eps m}{n-m} \int |\nabla (F'(u) \Delta F(u) ) |^2 d\x.
%	\end{align*}
%\end{enumerate}

%\vfil\eject
%  section linearized
%%%%%%%%%%%%%%%%%%%%%%%%%%%%%%%%%%%%%%%%%%%%%%%%%%%%%%%%%%%%%%%%

\section{Linearized stability and classification of instability types}\label{sec:linear}
In this section we will carry out the stability analysis of the system \eqref{pmi} in the $d$-dimensional hypercube $\Omega = [0,1]^d$ around the equilibrium solution. On a bounded domain with Neumann boundary conditions, the function $u$ satisfies the conservation law
\begin{equation}
	\int_\Omega u(\textbf{x},T) d\textbf{x} = \int_{\Omega} u_0(\textbf{x})d\textbf{x}, \hspace{0.4cm} \forall  \hspace{0.1cm} T>0, \hspace{0.1cm} \textbf{x}\in \Omega.
\end{equation}
Denote $a:= \int_\Omega u_0(\textbf{x})d\textbf{x}$. There is a unique trivial solution $(u,\gamma) = (a,a^m)$. We are interested in the linear stability analysis of the trivial solution via the perturbation method. To this end, for $\delta >0$, we let 
\begin{align*}
	u &= a + \delta \bar{u} + O(\delta^2), \\
	\gamma &= a^m + \delta \bar{\gamma} + O(\delta^2). 
\end{align*}
By binomial series expansion, omitting nonlinear terms  
\begin{align*}
	\delta \bar{u}_t &= \Delta \Bigg[ \frac{1}{a^m + \delta \bar{\gamma} + O(\delta^2) } \big( a + \delta \bar{u} + O(\delta^2) \big)^n \Bigg] + O(\delta^2) 
	\\
	&= a^{n-m} \ \Delta  \Bigg[\big(1 - \delta a^{-m}\bar{\gamma} \big) \big( 1 + \delta n a^{-1} \bar{u}  \big) \Bigg] + O(\delta^2) 
	\\
	&= \delta n a^{n-m-1} \Delta  \bar{u} - \delta a^{n-2m} \Delta  \bar{\gamma} + O(\delta^2). 
\end{align*}
Similarly, 
\begin{align*}
	\delta \bar{\gamma}_t &=  
		- \delta \bar{\gamma}  + \delta m a^{m-1} \bar{u} + O(\delta^2). 
\end{align*}
Thus we arrive at the linearized system:
\begin{align*}
	\bar{u}_t &= n a^{n-m-1}  \Delta  \bar{u} - a^{n-2m} \Delta  \bar{\gamma} , \\ 
	\bar{\gamma}_t &= m a^{m-1} \bar{u} -   \bar{\gamma}.
\end{align*}
The characteristic equation of the linearized operator is independent of $a$. Thus, we assume that $a=1$ without loss of generality, and write the linearized system in matrix form 
\begin{equation*}
	\frac{d}{dt} \begin{bmatrix}
		\bar{u} \\ \bar{\gamma}
	\end{bmatrix} = 
	\begin{bmatrix}
		n \Delta  & - \Delta  \\ 
		m   & -1
	\end{bmatrix} \ \begin{bmatrix}
		\bar{u} \\ \bar{\gamma}
	\end{bmatrix}.
\end{equation*}

For the linearized system with Neumann boundary conditions, we express the solution via a cosine Fourier series 
\begin{align*}
	\bar{u}(\textbf{x},t) &= \bar{u}(x_1,\cdots,x_d,t) = \sum_{k_1=0}^\infty \sum_{k_2=0}^\infty \cdots \sum_{k_d=0}^\infty \hat{u}_{k_1,k_2,\cdots,k_d}(t) \prod_{j=1}^d \cos(k_j \pi x_j),\\
	\bar{\gamma}(\textbf{x},t) &= \bar{\gamma}(x_1,\cdots,x_d,t) = \sum_{k_1=0}^\infty \sum_{k_2=0}^\infty \cdots \sum_{k_d=0}^\infty \hat{\gamma}_{k_1,k_2,\cdots,k_d}(t) \prod_{j=1}^d \cos(k_j \pi x_j).
\end{align*}
%where $\hat{u}_{k_1,k_2,\cdots,k_d}(t)$ and $\hat{\gamma}_{k_1,k_2,\cdots,k_d}(t)$ are given by \tcr{later}.\\\\
where the Fourier coefficients $\big(\hat{u}_{k_1,k_2,\cdots,k_d}(t), \hat{\gamma}_{k_1,k_2,\cdots,k_d}(t) \big)$ satisfy the following ordinary differential system; 
\begin{equation*}
	\frac{d}{dt} \begin{bmatrix}
		\hat{u}_{k_1, k_2, \cdots,k_d}  \\ \hat{\gamma}_{k_1, k_2, \cdots,k_d} 
	\end{bmatrix} = 
	\begin{bmatrix}
		-n \pi^2 (k_1^2+k_2^2+\cdots+k_d^2)  & \pi^2 (k_1^2+k_2^2+\cdots+k_d^2)  \\ 
		m   & -1
	\end{bmatrix} \ \begin{bmatrix}
		\hat{u}_{k_1, k_2, \cdots,k_d} \\ \hat{\gamma}_{k_1, k_2, \cdots,k_d} 
	\end{bmatrix}.
\end{equation*}

We next carry out the eigenvalue analysis for the coefficient matrix 
\begin{equation*}
	A_{\textbf{k}} := 	\begin{bmatrix} -n \pi^2 (k_1^2+k_2^2+\cdots+k_d^2)  & \pi^2 (k_1^2+k_2^2+\cdots+k_d^2) \\ m   & -1 \end{bmatrix}
\end{equation*}
The characteristic equation is given by 
\begin{equation*}
	\operatorname{ch}(A_{\textbf{k}}) = \lambda^2 + (n\pi^2(k_1^2+k_2^2+\cdots+k_d^2) + 1) \lambda + (n-m) \pi^2(k_1^2+k_2^2+\cdots+k_d^2) = 0,
\end{equation*}
and the two eigenvalues are 
\small
\begin{align*}
	\lambda_{\textbf{k},\pm} &= \frac{-(n\pi^2l  + 1) \pm \sqrt{(n\pi^2l  + 1)^2 - 4(n-m) \pi^2l }}{2} \\ &= \frac{-(n\pi^2l  + 1) \pm \sqrt{(n\pi^2l  - 1)^2 + 4m \pi^2l }}{2}.
\end{align*}
where from now on we denote $l = k_1^2+k_2^2+\cdots+k_d^2$. Note that the two eigenvalues are real for all $m\ge 0$ and $n\ge 0$. 
\\\\
Having derived the growth rates $\lambda_{\textbf{k},\pm}$ associated with each Fourier mode, $k=0,1,2,\dots$, we classify the instabilities according to the behavior of eigenvalues as a function of $\textbf{k}$. Note that for the real eigenvalues $\lambda_{\textbf{k},\pm}$, we have $\lambda_{\textbf{k},+} \ge \lambda_{\textbf{k},-} $ and the growth rate of $\lambda_{\textbf{k},+}$ determines its instability type. A Hadamard-type instability is characterized by unbounded growth rates at high frequencies, namely $\sup_{\textbf{k}\ge 0}\Re \lambda_{\textbf{k},+} = +\infty$, which indicates ill-posedness through arbitrarily fast amplification of small-scale Fourier modes. In contrast, a Turing-type instability occurs when the homogeneous mode is stable, $\Re \lambda_{0,+}\le 0$, while $\Re \lambda_{\textbf{k},+}>0$ for some $\textbf{k}\ne 0$, with $\sup_\textbf{k} \Re \lambda_{\textbf{k},+}<\infty$. Following the perspective of Miyazako, Hori, and Hara \cite{MiyazakoHoriHara2013}, Turing instability mechanisms can be distinguished by whether the dominant growth occurs at finite or arbitrarily large wave numbers, which are named Type~I and Type~II Turing instabilities. According to the location of the most unstable modes, in Type~I, there exists a finite wavenumber $\textbf{k}_*>0$ such that $\Re \lambda_{\textbf{k}_*,+}=\max_{\textbf{k}\ge 0}\Re \lambda_{\textbf{k},+}>0$, leading to intrinsic wavelength selection; whereas in Type~II, the instability is dominated by increasingly high-frequency modes, in the sense that $\Re \lambda_{\textbf{k},+} >0$ for all sufficiently large $l = k_1^2+k_2^2+\cdots+k_d^2$, and the maximal growth rate is attained only in the limit $l\to\infty$. The asymptotic stability of our model can be classified as follows depending on the parameter range of $(n,m)$.

\vspace{0.8em}
\begin{enumerate}[{\bf Case} 1.]
	\item If $n= m= 0$, then we have stability along a line. $\lambda_{\textbf{k},+} = 0$ and $\lambda_{\textbf{k},-} = -1$. 
    \vspace{0.8em}
    
	\item If $n=0$, $m>0$, \emph{Hadamard} instability.
	\begin{align*}
		\lambda_{\textbf{k},+} &= \frac{- 1 + \sqrt{1 + 4m \pi^2l}}{2} = \sqrt{m} O(\sqrt{l}) \\ \lambda_{\textbf{k},-} &= \frac{- 1 - \sqrt{1 + 4m\pi^2 l}}{2} < -1
	\end{align*}
    \vspace{0.8em}
	
	\item If $n>m\ge 0$, Stable.
	\small
	\begin{align*}
		\lambda_{\textbf{k},+} &= \frac{-(n\pi^2l + 1) + \sqrt{(n\pi^2l + 1)^2 - 4(n-m) \pi^2l}}{2} < 0 \\ \lambda_{\textbf{k},-}&= \frac{-(n\pi^2l + 1) - \sqrt{(n\pi^2(l + 1)^2 - 4(n-m) \pi^2l}}{2} < 0.
	\end{align*}
    \vspace{0.8em}
	
	\item If $ n = m > 0$, Stable along a line.
	\begin{align*}
		\lambda_{\textbf{k},+} &= \frac{-(n\pi^2l + 1) + \sqrt{(n\pi^2l + 1)^2 }}{2} = 0 \\ \lambda_{\textbf{k},-}&= \frac{-(n\pi^2l + 1) - \sqrt{(n\pi^2l + 1)^2 }}{2} = -1.
	\end{align*}
    \vspace{0.8em}
	
	\item If $m> n> 0$, \emph{Turing Type II} instability.
	\begin{align*}
		\lambda_{\textbf{k},+} &= \frac{-(n\pi^2l + 1) + \sqrt{(n\pi^2l + 1)^2 - 4(n-m) \pi^2l}}{2} > 0 \\ \lambda_{\textbf{k},-}&= \frac{-(n\pi^2l + 1) - \sqrt{(n\pi^2l+ 1)^2 - 4(n-m) \pi^2l}}{2} < 0.
	\end{align*}
    \vspace{0.8em}
\end{enumerate}
\begin{lemma}[Proof for Turing Type II for Case 5.]
	If $m>n>0$, then $\lambda_{\textbf{k},+} \uparrow$ as $\sqrt{l} \uparrow$ and $$ \lambda_{0,0,+} = 0 <  \lambda_{\textbf{k},+} < \frac{m-n}{n} .$$ 
\end{lemma}
\begin{proof}
	\begin{align*}
		\lambda_{\textbf{k},+} &= \frac{-(n\pi^2l + 1) + \sqrt{(n\pi^2l + 1)^2 + 4(m-n) \pi^2l}}{2} \\ &= \frac{2(m-n) \pi^2l}{n\pi^2l + 1 + \sqrt{(n\pi^2l + 1)^2 + 4(m-n) \pi^2l}} \\ & < \frac{(m-n) \pi^2l)}{n\pi^2l + 1} < \frac{m-n}{n}.
	\end{align*}
	From the characteristic equation $\operatorname{ch}(A_{\textbf{k}}) = 0$, $\lambda_{\textbf{k},+}$ satisfies
	\begin{equation*}
		\lambda_{\textbf{k},+}^2 + (n\pi^2l + 1)\lambda_{\textbf{k},+} + (n-m)\pi^2l = 0
	\end{equation*}
	By differentiating with $s = \sqrt{l}$, 
	\begin{equation*}
		2 \lambda_{\textbf{k},+} \frac{d \lambda_{\textbf{k},+}}{ds} + (n\pi^2s^2 + 1)\frac{d \lambda_{\textbf{k},+}}{ds} + 2n\pi^2s\lambda_{\textbf{k},+} + 2(n-m)\pi^2s = 0
	\end{equation*}
	\begin{equation*}
		\frac{d \lambda_{\textbf{k},+}}{ds} = 2n\pi^2s \cdot \frac{\frac{m-n}{n} - \lambda_{\textbf{k},+} }{2 \lambda_{\textbf{k},+} + n\pi^2s^2 + 1} > 0.
	\end{equation*}
\end{proof}

%\vfil\eject
%  section nonlinear stable
%%%%%%%%%%%%%%%%%%%%%%%%%%%%%%%%%%%%%%%%%%%%%%%%%%%%%%%%%%%%%%%%

\section{Nonlinear Analysis of Equilibration for \texorpdfstring{$n>m$}{n>m}} \label{sec:equilibration}
We will focus here on \eqref{pmi} for the case $n >m$. The aim is to show that in this regime, our solution $(u,\gamma)$ converges to the equilibrium as $t\rightarrow \infty$ in the $L^2$- sense. We start by deriving the energy structure of the system, exploiting a Lyapunov-type functional. 
In this section, we impose that the initial conditions be compatible with the boundary data, namely,  
\begin{equation}
\frac{\partial u_0}{\partial \nu} = 0 \hspace{0.2cm}\text{on}\hspace{0.1cm}\partial \Omega, \hspace{0.4cm} \frac{\partial \gamma_0}{\partial \nu} = 0 \hspace{0.2cm}\text{on}\hspace{0.1cm}\partial \Omega.
\label{initial-data-restrict}
\end{equation}

\begin{lemma}[Energy estimate.]\label{lem:energy} Let $\Omega \subset \mathbb{R}^d$ be an open and bounded subset of $\mathbb{R}^d$ for $d \geq 1$. Let $(u,\gamma)$ be a positive smooth solution to \eqref{pmi} with initial data $(u_0,\gamma_0)\in L^\infty(\Omega)\times L^\infty(\Omega)$ satisfying \eqref{initial-data-restrict}. Then, for $n, m >0$, $(u,\gamma)$ satisfies 
	\begin{equation}
		\frac{d}{dt}\int_{\Omega} F(u,\gamma) d\x + (n+1)\int_{\Omega} |\nabla \frac{u^n}{\gamma}|^2 d\x + \int_{\Omega} \frac{1}{\gamma^2} \Big((u^m)^{\frac{n+1}{m}}-\gamma^{\frac{n+1}{m}}\Big)(u^m-\gamma) d\x = 0.
	\end{equation}
    \label{energy}
	where 
	\begin{equation}
		F(u,\gamma) := \frac{u^{n+1}}{\gamma}+\frac{m}{n-m+1}\gamma^{\frac{n-m+1}{m}},
	\end{equation}
	is a convex functional near the equilibrium for $n>m$, and is additionally globally convex for $n> 1, n+1 >2m$.  %\tcb{Confusion: the functional is convex near equilibrium for $n >m $ and loses convexity in the other regime. Can we say anything about the convexity for $n>m$ but away from the equilibrium? Or is that not needed?}
\end{lemma}
\begin{proof}
	Multiplying $\eqref{pmi}_1$ by $(n+1)\frac{u^n}{\gamma}$ and integrating by parts, we get 
	\begin{equation}
		\int_{\Omega} \frac{1}{\gamma} \partial_t u^{n+1}d\x = -(n+1)\int_\Omega |\nabla \frac{u^n}{\gamma}|^2 d\x,
		\label{estimate1}
	\end{equation}
	as we have zero Neumann boundary conditions. Moreover, 
	\begin{align}
		\partial_t \Big(\frac{u^{n+1}}{\gamma}\Big) &= \frac{1}{\gamma}\partial_t u^{n+1}-\frac{u^{n+1}}{\gamma^2}\partial_t \gamma 
		= \frac{1}{\gamma} \partial_t u^{n+1}-\frac{u^{n+1}}{\gamma^2}(u^m-\gamma).
		\label{time-der}
	\end{align}
	To control the last term of the above equation, we multiply $\eqref{pmi}_2$ by $\gamma^{\Tilde {\alpha}}$, where $\Tilde{\alpha} >0$, to obtain
	\begin{equation}
		\frac{1}{\Tilde{\alpha}+1} \partial_t \gamma^{\Tilde{\alpha}+1} = \gamma^{\Tilde{\alpha}}(u^m-\gamma).
		\label{estimate2}
	\end{equation}
	Adding \eqref{time-der} and \eqref{estimate2}
	\begin{align*}
		\partial_t \Big(\frac{u^{n+1}}{\gamma}+\frac{1}{\Tilde{\alpha}+1} \gamma^{\Tilde{\alpha}+1}\Bigg) &= \frac{1}{\gamma}\partial_t u^{n+1}-\frac{u^{n+1}}{\gamma^2}(u^m-\gamma)+ \gamma^{\Tilde{\alpha}}(u^m - \gamma)\\
		%&= \frac{1}{\gamma} \partial_t u^{n+1}-\frac{1}{\gamma^2}\Big(u^{n+1}-\gamma^{\Tilde{\alpha}+2}\Big)(u^m-\gamma)\\
		&= \frac{1}{\gamma} \partial_t u^{n+1}-\frac{1}{\gamma^2}\Bigg((u^m)^{\frac{n+1}{m}}-\gamma^{\Tilde{\alpha}+2}\Bigg)(u^m-\gamma).
	\end{align*}
	Choosing $\Tilde{\alpha} = \frac{n+1}{m}-2$, we see that the last term above is non-positive for any positive exponents $m,n$. \\
    We then have after integrating 
\begin{equation}
\begin{aligned}
		\frac{d}{dt} \int_\Omega \Bigg(\frac{u^{n+1}}{\gamma} &+\frac{m}{n-m+1}\gamma^{\frac{n-m+1}{m}}\Bigg) d\x 
	\\
	&= -(n+1)\int_{\Omega} |\nabla \frac{u^n}{\gamma}|^2 d\x -\int_{\Omega} \frac{1}{\gamma^2} \Bigg((u^m)^{\frac{n+1}{m}}-\gamma^{\frac{n+1}{m}}\Bigg)(u^m-\gamma)d\x.
\end{aligned}
\end{equation}
	Therefore, for the functional $F(u,\gamma) = \int_\Omega \frac{u^{n+1}}{\gamma} + \frac{m}{n-m+1} \gamma^{\frac{n-m+1}{m}}dx$, we have 
	\begin{equation}
		\frac{d}{dt}F(u,\gamma) + (n+1)\int_{\Omega} |\nabla \frac{u^n}{\gamma}|^2 d\x + \int_\Omega \frac{1}{\gamma^2}\Bigg((u^m)^{\frac{n+1}{m}}-\gamma^{\frac{n+1}{m}}\Bigg)(u^m-\gamma) d\x =0. 
	\end{equation}
    The functional $F$ is uniquely minimized at the equilibrium $(\overline{u},\overline{\gamma})= (a,a^m)$. To see this, we first assume without loss of generality that $|\Omega|=1$. Applying Young's inequality $ab\leq \frac{a^p}{p}+\frac{b^q}{q}$ with $p= \frac{n+1}{n-m+1}$, $q= \frac{n+1}{m}$, $a= \frac{u^{n-m+1}}{\gamma^{\frac{n-m+1}{n+1}}}$, $b = \gamma^{\frac{n-m+1}{n+1}}$, we obtain 
    \begin{equation}
        F(u,\gamma) \geq \frac{n+1}{n-m+1} \int_{\Omega} u^{n-m+1}d\x,
    \end{equation}
    where the right-hand side is strictly convex for $n>m$. Minimizing over the affine constraint set $\mathcal{A} = \Big\{u >0; \int_\Omega u(\x,t)d\x = a \Big\}$ implies that there is at most one minimizer. Applying Jensen's inequality 
    \begin{equation}
        \int_{\Omega} u^{n-m+1}\hspace{0.1cm}d\x \geq \Bigg(\int_{\Omega} u \hspace{0.1cm}d\x\Bigg)^{n-m+1} = a^{n-m+1},
    \end{equation}
    that is, the unique constrained minimizer is attained at $(\overline{u}, \overline{\gamma}) = (a,a^m)$. We can further verify by finding the first and second variations of the constrained problem 
    \begin{equation}
        \min_{(u,\gamma)} \hspace{0.1cm}F(u,\gamma) \hspace{0.2cm}\text{subject to} \hspace{0.2cm}  \int_{\Omega}u(\x,t)d\x = a.
    \end{equation}
    The Lagrangian is given by 
    \begin{equation*}
        \mathcal{L} = \int_{\Omega} \frac{u^{n+1}}{\gamma} +\frac{m}{n-m+1}\gamma^{\frac{n-m+1}{m}} -\lambda (u-a) \hspace{0.1cm}d\x.
    \end{equation*}
    The first variation gives the critical point $(\bar{u},\bar{\gamma})=(a,a^m)$. To show that $(\bar{u},\bar{\gamma}) = (a,a^m)$ is indeed a minimum, we find the second variation $\delta^2 F$ of $F$. To this end, we consider the perturbations $u_\varepsilon = a +\varepsilon \varphi$, $\gamma_\varepsilon = a^m + \varepsilon \psi$, where $\varphi, \psi$ are smooth functions independent of $\varepsilon$ with $\int_{\Omega}\varphi dx =0$. Then the second variation at the equilibrium $(\bar{u},\bar{\gamma}) = (a,a^m)$ is given by 
    \begin{align*}
        \delta^2F(a,a^m) &= \int_{\Omega} n(n+1) a^{n-m+1}\varphi^2-2(n+1)a^{n-2m}\varphi \psi + \frac{n+1}{m} a^{n-3m+1}\psi^2 d\x \\
        &= (n+1)a^{n-m-1}\int_{\Omega} n \Big(\varphi-\frac{a^{1-m}\psi}{n}\Big)^2 + \Big(\frac{n-m}{nm}\Big)\Big(a^{2-2m}\psi^2\Big)d\x,
    \end{align*}
    from which we have $\delta^2 F(a,a^m) > 0$ exactly when $n>m$.\\
    Away from the equilibrium, the second variation of $F$ is given by 
    \begin{align*}
       \delta^2 F(u,\gamma) &= \int_{\Omega} n(n+1)\frac{u^{n-1}}{\gamma}\varphi^2 -2(n+1)\frac{u^n}{\gamma^2}\varphi\psi +\Big(2\frac{u^{n+1}}{\gamma^3}+ \frac{n-3m+1}{m}\gamma^{\frac{n-3m+1}{m}}\Big)\psi^2 d\x \\
       &= \int_{\Omega} n(n+1)\frac{u^{n-1}}{\gamma}\Big(\varphi-\frac{u}{n\gamma}\psi\Big)^2 + \Big(\frac{n-1}{n}\frac{u^{n+1}}{\gamma^3}+\frac{n-2m+1}{m}\gamma^{\frac{n-3m-1}{m}}\Big)\psi^2 d\x,
    \end{align*}
    that is, $\delta^2 F(u,\gamma)>0$ provided that $n >1 $ and $n+1 > 2m$.
	 \end{proof}
	% \tcb{Note that near the equilibrium $(a,a^m)$, the functional $F$ is convex for $n > m > 0$. 
	% 	\begin{equation*}
	% 		\nabla^2 F(a,a^m) = 	\begin{bmatrix} n(n+1)a^{n-m-1} & -(n+1)a^{n-2m} \\\\ -(n+1)a^{n-2m}   & \frac{n+1}{m}a^{n-3m+1} \end{bmatrix}
	% 	\end{equation*}
	% 	then 
	% 	\begin{equation*}
	% 		\Tr \nabla^2 F(a,a^m)= n(n+1)a^{n-m-1}+\frac{n+1}{m}a^{n-3m+1} >0 (\tcr{a>0})
	% 	\end{equation*}
	% 	\begin{equation*}
	% 		\det \nabla^2 F(a,a^m)= (n+1)^2 a^{2n-4m}(\frac{n}{m}-1)   
	% 	\end{equation*}
	% 	we see that $\det \nabla^2 F(a,a^m)>0$ if and only if $n >m$. In particular, $F$ loses convexity near equilibrium in the unstable regime $ n < m$.}\\\\
	Next, we consider a comparison argument for our model based on the maximum principle \cite{chueh1977positively}. To do that, we first apply a change of variable $(u,\gamma) \to (\sigma,\gamma)$ where $\sigma := \frac{u^n}{\gamma}$, to obtain the following reaction-diffusion system 
	\begin{equation} \label{eq:sigma}
		\tag{\tcb{RED}}
		\begin{cases}
			\partial_t \sigma = n \sigma^{1-{1\over n}} \gamma^{- {1\over n}}\Delta \sigma + \sigma \gamma^{m-n \over n} \Big(\gamma^{n-m \over n}-\sigma^{m\over n}  \Big), \\
			\partial_t \gamma = \gamma^{m\over n} \Big( \sigma^{m\over n} - \gamma^{n-m \over n} \Big).
		\end{cases}
	\end{equation}
	with the zero Neumann boundary condition
	\begin{equation}
		\frac{\partial}{\partial \nu}\sigma(\textbf{x},t) =0.
	\end{equation}
	Here, a critical line for the relaxation terms is $\sigma = \gamma^{n-m \over m}$. We draw the critical line on $\gamma\sigma$-plane for two cases, $n>m$ and $n<m$. We first assume that our initial values are bounded and bounded away from zero. $$ M_l\le \min(\gamma_{0}) \le \max(\gamma_{0}) \le M_u, \quad \Tilde{M_l} \le \min(\sigma_{0}) \le \max(\sigma_{0}) \le \Tilde{M_u} $$
	\begin{proposition}
		Assume $n >m$. Let $(\gamma,\sigma)$ be a positive smooth solution to \eqref{eq:sigma} with $(\gamma_0,\sigma_0) \in L^\infty(\Omega)\times L^\infty(\Omega)$. Then, there exist positive constants $\Gamma_{-}, \Gamma_{+} >0$ such that $\Sigma_{-} = (\Gamma_{-})^{\frac{n-m}{m}}, \Sigma_{+} = (\Gamma_{+})^{\frac{n-m}{m}}$ where $\gamma_0, \sigma_0 \in D := [\Gamma_{-}, \Gamma_{+}]\times [\Sigma_{-}, \Sigma_{+}]$ and $D$ is a positively invariant region under the flow of the system \eqref{eq:sigma}.
		
		\label{max-principle}
		%for every $t \in [T, \infty), \hspace{0.1cm}\textbf{x}\in \Omega.$
	\end{proposition}
	%\tcr{still not sure how to better state this theorem.} 
	\begin{remark}
		It follows that since $u = (\sigma \gamma)^{\frac{1}{n}}$, we also have 
		\begin{equation*}
			\Gamma_{-}^{\frac{1}{m}} \leq u(\textbf{x},t) \leq \Gamma_{+}^{\frac{1}{m}}.
		\end{equation*}
		for every $t \in [T, \infty), \hspace{0.1cm}\textbf{x}\in \Omega.$
	\end{remark}
	\begin{proof}
		For two positive constants $\Gamma_{-}$ and $\Gamma_{+}$, we define a rectangular domain  $$ D := \left\{ (s,\Tilde{s}) \in \R_+^2 \ | \ \Gamma_{-} \le s \le \Gamma_{+}, \ \Sigma_{-} \le \Tilde{s} \le \Sigma_{+} \right\}.  $$ And the functions $(\gamma,\sigma)(x,t)$ are in the domain $D$ at time $T$ if they satisfy $$\Gamma_{-} \le \min_x(\gamma)(T) \le \max_x(\gamma)(T) \le \Gamma_{+}, \quad \Sigma_{-} \le \min_x(\sigma)(T) \le \max_x(\sigma)(T) \le \Sigma_{+}. $$
		We claim that $D$ is an invariant region, that is, if $(\gamma,\sigma)$ is in the domain $D$ at time $T$, then it is trapped in $D$ for every $t \ge T$. It suffices to show that for every $(\gamma,\sigma) \in \partial D$, we have $$ (\partial_t \gamma, \partial_t \sigma) \cdot \nu \le 0,$$ where $\nu$ is an outward normal vector of $D$ on $(\gamma,\sigma)$. We divide the boundary of $D$ into 4 subdomains, 
		\begin{align*}
			I_1 &= \left\{\Gamma_{-}\right\} \times [\Sigma_{-}, \Sigma_{+}], \\ 
			I_2 &= [\Gamma_{-}, \Gamma_{+}] \times \left\{ \Sigma_{+} \right\}, \\
			I_3 &= \left\{\Gamma_{+}\right\} \times [\Sigma_{-}, \Sigma_{+}], \\ 
			I_4 &= [\Gamma_{-}, \Gamma_{+}] \times \left\{ \Sigma_{-} \right\}.
		\end{align*}
		The normal vectors on each domain are $(-1,0)$, $(0,1)$, $(1,0)$, and $(0,-1)$ respectively. Along the line $I_1$, it is placed above the critical line $\sigma = \gamma^{n-m \over m}$ and then, $\sigma^{m\over n} - \gamma^{n-m \over n} \ge 0$.
		$$ (\partial_t v, \partial_t \sigma) \cdot \nu = -\partial_t \gamma = -\gamma^{m\over n} \Big( \sigma^{m\over n} - \gamma^{n-m \over n} \Big) \le 0.$$
		Along the line $I_2$, again it satisfies $\sigma^{m\over n} - \gamma^{n-m \over n} \ge 0$. Moreover, since $\Sigma_{+}$ is the local maximum of the trapped function $\sigma$ on the space domain $\Omega$, $\Delta \sigma \le 0$. Therefore, 
		$$ (\partial_t \gamma, \partial_t \sigma) \cdot \nu = \partial_t \sigma =  n \sigma^{1-{1\over n}} \gamma^{- {1\over n}}\Delta \sigma + \sigma \gamma^{m-n \over n} \Big( \gamma^{n-m \over n}-\sigma^{m\over n}  \Big) \le 0.$$ Similarly, along the line $I_3$, $\sigma^{m\over n} - \gamma^{n-m \over n} \le 0$ implies $$ (\partial_t \gamma, \partial_t \sigma) \cdot \nu = \partial_t \gamma = \gamma^{m\over n} \Big( \sigma^{m\over n} - \gamma^{n-m \over n} \Big) \le 0.$$ Along the line $I_4$, $\sigma^{m\over n} - \gamma^{n-m \over n} \le 0$ and $\Delta \sigma \ge 0$ imply $$ (\partial_t \gamma, \partial_t \sigma) \cdot \nu = - \partial_t \sigma =  - n \sigma^{1-{1\over n}} \gamma^{- {1\over n}}\Delta \sigma - \sigma \gamma^{m-n \over n} \Big( \gamma^{n-m \over n}-\sigma^{m\over n}  \Big) \le 0.$$ Therefore, $D$ is an invariant region and 
		
		\begin{equation*}
			\Gamma_{-} \leq \gamma(\textbf{x},t) \leq \Gamma_{+},
		\end{equation*}
		\begin{equation*}
			\Sigma_{-}\leq \sigma(\textbf{x},t) \leq \Sigma_{+},
		\end{equation*}
		%\begin{equation*}
		%   A^{\frac{1}{m}} \leq u(\textbf{x},t) \leq B^{\frac{1}{m}}.
		%\end{equation*}
		%for every $t \in [T, \infty).$
	\end{proof}
	We next derive an estimate for \eqref{eq:sigma}. 
	\begin{lemma}[Estimate for $\sigma$.] Assume that $(u, \gamma)$ is a solution of \eqref{pmi} emanating from initial  data $(u_0,\gamma_0)\in L^\infty(\Omega)\times L^\infty(\Omega)$ satisfying \eqref{initial-data-restrict}. Define $\sigma := \frac{u^n}{\gamma}$. Then for any $n>m$, the following estimate holds,
		\begin{equation}
			\label{estimate:sigma}
			\frac{d}{dt}\int_\Omega |\nabla \sigma|^2 d\x + n\int_{\Omega}\frac{u^{n-1}}{\gamma} |\Delta \sigma|^2 d\x \leq \frac{1}{n}\int_{\Omega} \sigma^{1+\frac{1}{n}} \gamma^{\frac{1}{n}-2} \big(\partial_t \gamma\big)^2 d\x.
		\end{equation}   
	\end{lemma}
	\begin{proof}
		By multiplying $\eqref{eq:sigma}_1$ by $\Delta \sigma$, we obtain 
		\begin{align*}
			\int_\Omega \Delta \sigma \partial_t \sigma &= \int_\Omega n \sigma^{1-\frac{1}{n}}\gamma^{-\frac{1}{n}} |\Delta \sigma|^2 d\x + \int_\Omega \sigma \gamma^{\frac{m-n}{n}}\big(\gamma^{\frac{n-m}{m}}-\sigma^{\frac{m}{n}}\big)\Delta \sigma d\x\\
			&= n\int_{\Omega} |\Delta \sigma|^2 \sigma^{1-\frac{1}{n}}\gamma^{-\frac{1}{n}} d\x - \int_{\Omega} \frac{\sigma}{\gamma}\Delta \sigma (\partial_t \gamma) d\x 
		\end{align*}
		integrating by parts we arrive at 
		\begin{align*}
			\frac{d}{dt}\frac{1}{2} \int_{\Omega} |\nabla \sigma|^2 d\x + n \int_{\Omega}|\Delta \sigma|^2 \sigma^{1-\frac{1}{n}} \gamma^{-\frac{1}{n}}d\x &= \int_{\Omega} \frac{\sigma}{\gamma}\Delta \sigma (\partial_t \gamma) d\x \\
			&= \int_\Omega \Big(\sigma^{\frac{n+1}{2n}}\gamma^{\frac{1-2n}{2n}}\Big)\partial_t \gamma \Big(\sigma^{\frac{n-1}{2n}}\gamma^{-\frac{1}{2n}}\Big)\Delta \sigma d\x \\
			&\leq \frac{1}{2n}\int_{\Omega} \sigma^{\frac{n+1}{n}}\gamma^{\frac{1-2n}{n}} (\partial_t \gamma)^2 d\x + \frac{n}{2}\int_\Omega \sigma^{\frac{n-1}{n}}\gamma^{-\frac{1}{n}} |\Delta \sigma|^2 d\x, 
		\end{align*}
		absorbing the second term above to the left-hand side, we obtain 
		\begin{equation}
			\frac{d}{dt}\int_\Omega |\nabla \sigma|^2 d\x + n \int_\Omega \frac{u^{n-1}}{\gamma} |\Delta \sigma|^2 d\x \leq \frac{1}{n}\int_{\Omega}\sigma^{1+\frac{1}{n}}\gamma^{\frac{1}{n}-2}(\partial_t \gamma)^2 d\x. 
		\end{equation}
		Using $\eqref{eq:sigma}_2$ we further get 
		\begin{equation}
			\label{estimate:sigma2}
			\frac{d}{dt}\int_\Omega |\nabla \sigma|^2 d\x + n \int_\Omega \frac{u^{n-1}}{\gamma} |\Delta \sigma|^2 d\x \leq \frac{1}{n}\int_{\Omega} \frac{u^{n+1}}{\gamma^3} |u^m-\gamma|^2 d\x.
		\end{equation}
	\end{proof}
	The main theorem in this section is the following.
	\begin{theorem}
		Let $\Omega \subset \mathbb{R}^d$ be an open bounded subset for $d \geq 1$. Let $(u,\gamma)$ be a positive smooth solution to \eqref{pmi} 
		with initial data $(u_0,\gamma_0)\in L^\infty(\Omega)\times L^\infty(\Omega)$ satisfying \eqref{initial-data-restrict}. Then, if $n >m$, $(u,\gamma)$ converges to the equilibrium state in the $L^2$-sense as $t\rightarrow \infty$. That is, 
		\begin{equation}
			\lim_{t\rightarrow \infty}\Big\|\big(u(\textbf{x},t), \gamma(\textbf{x},t)\big)-(1,1)\Big\|_{L^2(\Omega)} = 0.
		\end{equation}
        \label{main-theorem}
	\end{theorem}
	Before proving Theorem \ref{main-theorem}, we first state some additional results. 
	\begin{lemma}
		Let $(u,\gamma)$ be a solution to \eqref{pmi} with $(u_0,\gamma_0) \in L^\infty(\Omega)\times L^\infty(\Omega)$ satisfying \eqref{initial-data-restrict} and let $n>m$. Define 
		\begin{equation}
			g(t):= \frac{\big(\Gamma_{+}\big)^{\frac{n+1}{m}}}{n\big(\Gamma_{+}\big)^3} \int_\Omega |u^m-\gamma|^2 d\x
			\label{defintion:g}
		\end{equation}
		There is a constant $C_{n,\Gamma_{-},\Gamma_{+}} > 0$, depending on $n$ and $\Gamma_{-}, \Gamma_{+}$ such that 
		\begin{equation}
			\int_0^\infty g(s) ds \leq C_{n,\Gamma_{-},\Gamma_{+}}.
		\end{equation}
		\label{bound-g}
	\end{lemma}
	\begin{proof}
		Recall the energy estimate from Lemma \ref{energy}
		\begin{equation*}
			\frac{d}{dt}\int_{\Omega}\Bigg(\frac{u^{n+1}}{\gamma}+\frac{m}{n-m+1}\gamma^{\frac{n-m+1}{m}}\Bigg) d\x + (n+1)\int_{\Omega} |\nabla \frac{u^n}{\gamma}|^2 d\x + \int_{\Omega}\frac{1}{\gamma^2} \Bigg((u^m)^{\frac{n+1}{m}}-\gamma^{\frac{n+1}{m}}\Bigg)(u^m-\gamma)d\x.
		\end{equation*}
		Note that by the mean value theorem we have 
		\begin{equation*}
			\int_{\Omega}\frac{1}{\gamma^2}|u^m-\gamma|^2 d\x \leq \int_\Omega \frac{1}{\gamma^2} \Bigg((u^m)^{\frac{n+1}{m}}-\gamma^{\frac{n+1}{m}}\Bigg)(u^m-\gamma)d\x.
		\end{equation*}
		Integrating in time and using Proposition \ref{max-principle},  
		\begin{multline*}
			\int_\Omega \Bigg(\frac{u^{n+1}}{\gamma}+\frac{m}{n-m+1}\gamma^{\frac{n-m+1}{m}}\Bigg)d\x \Bigg|_{t=T} + (n+1)\int_0^T \int_{\Omega} |\nabla \frac{u^n}{\gamma}|^2 d\x ds \\+ \int_0^T \int_{\Omega} \frac{1}{\big(\Gamma_{-}\big)^2} |u^m-\gamma|^2 d\x ds \leq \int_{\Omega} \frac{u_0^{n+1}}{\gamma_0}+\frac{m}{n-m+1} \gamma_0^{\frac{n-m+1}{m}}d\x = \Tilde{C},
		\end{multline*}
		for any positive time $T>0$, so that, 
		\begin{equation*}
			\int_0^\infty g(\tau) d\tau \leq \limsup_{T \to \infty} \int_0^T g(\tau) d\tau \leq \frac{n\Tilde{C}(\Gamma_{-})^2(\Gamma_{+})^3}{(\Gamma_{+})^{\frac{n+1}{m}}} : = C_{n, \Gamma_{-},\Gamma_{+}}.
		\end{equation*} 
	\end{proof}
	\begin{lemma}
		Let $(u,\gamma)$ be a solution to \eqref{pmi} with $(u_0,\gamma_0)\in L^\infty(\Omega)\times L^\infty(\Omega)$ satisfying \eqref{initial-data-restrict}. Assume $n >m$ and let $\sigma := \frac{u^n}{\gamma}$. Define $\varphi(t) := \int_{\Omega} |\nabla \sigma|^2 d\x$. Then $\varphi(t)$ satisfies 
		$\lim_{t\rightarrow \infty} \varphi(t) =0.$
		\label{asympt-sigma}
	\end{lemma}
	\begin{proof}
		By Lemma \ref{estimate:sigma}, we have the estimate 
		\begin{equation*}
			\frac{d}{dt} \int_{\Omega} |\nabla \sigma|^2 d\x + n \int_{\Omega}\frac{u^{n-1}}{\gamma} |\Delta \sigma|^2 d\x \leq \frac{1}{n}\int_{\Omega} \frac{u^{n+1}}{\gamma^3}|u^m-\gamma|^2 d\x, 
		\end{equation*}
		since $n >m$, we invoke Proposition  \ref{max-principle} to obtain 
		\begin{equation*}
			\frac{d}{dt} \int_{\Omega} |\nabla \sigma|^2 d\x + n\int_{\Omega} \frac{(\Gamma_{-})^{\frac{n-1}{m}}}{(\Gamma_{-})} |\Delta \sigma|^2 d\x \leq \frac{1}{n }\int_{\Omega}\frac{(\Gamma_{+})^{\frac{n+1}{m}}}{(\Gamma_{+})^3}|u^m-\gamma|^2 d\x,  
		\end{equation*}
		applying the Poincar\'e inequality to the second term on the left-hand side:
		\begin{equation*}
			\frac{d}{dt}\int_{\Omega}|\nabla \sigma|^2 d\x +n C\frac{(\Gamma_{-})^{\frac{n-1}{m}}}{(\Gamma_{-})} \int_{\Omega}|\nabla \sigma|^2 d\x \leq \frac{1}{n }\int_{\Omega}\frac{(\Gamma_{+})^{\frac{n+1}{m}}}{(\Gamma_{+})^3}|u^m-\gamma|^2 d\x := g(t).
		\end{equation*}
		So that we arrive at the differential inequality 
		\begin{equation}
			\frac{d}{dt} \varphi(t) + C_1 \varphi(t) \leq g(t),
		\end{equation}
		where $C_1$ depends on $A, n$ and the Poincar\'e constant. Multiplying both sides by $e^{C_1t}$ 
		\begin{align*}
			\varphi(t) &\leq e^{-C_1t}\varphi(0) + e^{-C_1t}\int_0^t e^{C_1\tau}g(\tau)d\tau \\
			%&= e^{-C_1t}\varphi(0)+ e^{-C_1t} \Bigg(\int_0^{t/2} g(\tau) e^{C_1 \tau}d\tau + \int_{t/2}^t g(\tau) e^{C_1\tau} d\tau\Bigg)\\
			&\leq e^{-C_1t} \varphi(0) + e^{-C_1t/2}\int_0^{t/2}g(\tau)d\tau + e^{-C_1t}\int_{t/2}^t g(\tau)e^{C_1 \tau} d\tau\\
			&\leq e^{-C_1t}\varphi(0)+C_{n,\Gamma_{-},\Gamma_{+}}e^{-C_1t/2} + e^{-C_1t}\int_{t/2}^t g(\tau) e^{C_1\tau} d\tau \\
			&\leq e^{-C_1t}\varphi(0) +C_{n,\Gamma_{-},\Gamma_{+}}e^{-C_1t/2} + \int_{t/2}^\infty g(\tau) d\tau
		\end{align*}
		thus 
		\begin{equation*}
			\lim_{t\rightarrow \infty} \varphi(t) = \lim_{t\rightarrow \infty} \int_{\Omega} |\nabla \sigma|^2 d\x =0. 
		\end{equation*}
	\end{proof}
	The next corollary is a consequence of Lemma \ref{asympt-sigma}.
	\begin{corollary}\label{lemma-conv-ugamma}
		Assume that $(u,\gamma)$  is a solution to \eqref{pmi} and that $n >m$. Let $\sigma := \frac{u^n}{\gamma}$. Then, 
		\begin{itemize}
			\item[(i)] $\lim_{t\rightarrow \infty} \int_{\Omega} |\nabla \gamma|^2 d\x =0 $, and
			\item[(ii)] $\lim_{t\rightarrow \infty} \int_{\Omega} |\nabla u|^2 d\x =0 $
		\end{itemize}
	\end{corollary}
	
		\begin{proof}
			Consider $\eqref{pmi}_2$ 
			\begin{equation*}
				\partial_t \gamma = -\gamma + u^m = -\gamma + \sigma^{\frac{m}{n}} \gamma^{\frac{m}{n}},
			\end{equation*}
			where  $\sigma = \frac{u^n}{\gamma}$. We multiply both sides by $\gamma^{-\frac{m}{n}}$
%			\begin{equation*}
%				\frac{1}{1-\frac{m}{n}}\partial_t \gamma^{1-\frac{m}{n}}= -\gamma^{-\frac{m}{n}+1} + \sigma^{\frac{m}{n}},  
%			\end{equation*}
			and take the gradient to obtain
			\begin{equation*}
				\frac{1}{1-\frac{m}{n}}\partial_t \nabla \gamma^{1-\frac{m}{n}} = -\nabla \gamma^{1-\frac{m}{n}} + \nabla \sigma ^{\frac{m}{n}},
			\end{equation*}
			A further multiplication by $\nabla \gamma^{1-\frac{m}{n}}$ and integration over $\Omega$ yields
%			\begin{equation*}
%				\frac{d}{dt}\frac{1}{2}\frac{1}{(1-\frac{m}{n})}\int_{\Omega} |\nabla \gamma^{1-\frac{m}{n}}|^2 d\x + \int_{\Omega} |\nabla \gamma^{1-\frac{m}{n}}|^2 d\x = \int_{\Omega} \nabla \gamma^{1-\frac{m}{n}}\nabla \sigma^{\frac{m}{n}}d\x
%			\end{equation*}
			\begin{align*}
				\frac{d}{dt} \int_{\Omega}|\nabla \gamma^{1-\frac{m}{n}}|^2 d\x &+ 2\big(1-\frac{m}{n}\big)\int_{\Omega} |\nabla \gamma^{1-\frac{m}{n}}|^2 d\x = 2\big(1-\frac{m}{n}\big)\int_{\Omega} \nabla \gamma^{1-\frac{m}{n}}\nabla \sigma^{\frac{m}{n}} d\x \\
				&\leq \Bigg(4(1-\frac{m}{n})^2 \int_\Omega |\nabla \gamma^{1-\frac{m}{n}}|^2 d\x\Bigg)^{1/2} \Bigg(\int_\Omega |\nabla \sigma^{\frac{m}{n}}|^2 d\x\Bigg)^{1/2} \\
				&\leq 2(1-\frac{m}{n})^2 \int_{\Omega}|\nabla \gamma^{1-\frac{m}{n}}|^2 d\x + \frac{1}{2}\int_\Omega |\nabla \sigma^{\frac{m}{n}}|^2 d\x,
			\end{align*}
			so that we get 
			\begin{equation}
				\frac{d}{dt}\int_\Omega |\nabla \gamma^{1-\frac{m}{n}}|^2 d\x + \frac{2m}{n}(1-\frac{m}{n}) \int_\Omega |\nabla \gamma^{1-\frac{m}{n}}|^2 d\x \leq \frac{1}{2} \int_\Omega |\nabla \sigma^{\frac{m}{n}}|^2 d\x.
			\end{equation}
			% taking the gradients inside the norm 
			% \begin{equation}
			% 	\frac{d}{dt}\int_\Omega |\nabla \gamma^{1-\frac{m}{n}}|^2 d\x + \frac{2m}{n}(1-\frac{m}{n})\int_\Omega \big|(1-\frac{m}{n})\gamma^{-\frac{m}{n}}\nabla \gamma\big|^2 d\x \leq \frac{1}{2} \int_\Omega |\frac{m}{n}\sigma^{\frac{m}{n}-1}\nabla \sigma|^2 d\x
			% \end{equation}
			% Applying Proposition  \ref{max-principle} we obtain 
			% \begin{equation*}
			% 	\frac{d}{dt}\int_\Omega |\nabla \gamma^{1-\frac{m}{n}}|^2 d\x + \frac{2m}{n}(1-\frac{m}{n})^3 (\Gamma_{-})^{-\frac{2m}{n}}\int_\Omega |\nabla \gamma|^2 d\x \leq \frac{1}{2} \big(\frac{m}{n}\big)^2 (\Gamma_{+})^{-\frac{2(n-m)^2}{nm}}\int_\Omega |\nabla \sigma |^2 d\x 
			% \end{equation*}
			% that is 
			% \begin{equation*}
			% 	\frac{d}{dt}\int_\Omega |\nabla \gamma^{1-\frac{m}{n}}|^2 d\x +\Tilde{C_1}\int_\Omega |\nabla \gamma|^2 d\x \leq \Tilde{C_2}\int_\Omega |\nabla \sigma|^2 d\x 
			% \end{equation*}
			% where the constants $\Tilde{C_1}, \Tilde{C_2}$ depend on $n,m, \Gamma_{-}$ and $\Gamma_{+}$. By Lemma \ref{asympt-sigma}, we have 
			% \begin{equation*}
			% 	\lim_{t\rightarrow\infty} \int_\Omega |\nabla \sigma|^2 d\x =0,
            
			% \end{equation}
            \noindent
            Using Proposition \ref{max-principle} we obtain 
            \begin{equation*}
                \frac{d}{dt}\int_{\Omega}|\nabla \gamma^{1-\frac{m}{n}}|^2dx + C_1\int_{\Omega}|\nabla \gamma^{1-\frac{m}{n}}|^2 dx \leq \frac{1}{2} (\frac{m}{n})^2(\Gamma_{+})^{\frac{-2(n-m)^2}{nm}} \int_{\Omega}|\nabla \sigma|^2 dx := C_2 \varphi(t),
            \end{equation*}
            where $C_1 >0$ and $C_2 > 0$ are constants. Denote $\psi(t):= \int_{\Omega} |\nabla \gamma^{1-\frac{m}{n}}|^2 dx$. Then we get the differential inequality 
            \begin{equation*}
                \frac{d}{dt} \psi(t) + C_1 \psi(t) \leq C_2 \varphi(t),
            \end{equation*}
            following the argument in the proof of Lemma \ref{asympt-sigma} and the fact that $\lim_{t\rightarrow \infty} \varphi(t)=0$, we obtain 
            \begin{equation*}
                \lim_{t\rightarrow \infty}\int_{\Omega}\psi(t) = \lim_{t\rightarrow \infty} \int_{\Omega}|\nabla \gamma^{1-\frac{m}{n}}|^2 dx =0. 
            \end{equation*}
            Moreover, from Proposition \ref{max-principle} we get 
            \begin{equation*}
                \int_{\Omega}|\nabla \gamma|^2 dx \leq \frac{n}{n-m}(\Gamma_{-})^{\frac{m}{n}}\int_{\Omega} |\nabla \gamma^{1-\frac{m}{n}}|^2 dx = C_3 \int_{\Omega}|\nabla \gamma^{1-\frac{m}{n}}|^2 dx,
            \end{equation*}
            where $C_3 >0$ is a positive constant. This gives 
            \begin{equation*}
                \lim_{t\rightarrow \infty} \int_{\Omega}|\nabla \gamma|^2 dx =0,
            \end{equation*}
			which proves (i). Now to prove (ii), we first note that since $\sigma = \frac{u^n}{\gamma}$ we have 
			\begin{equation*}
				\nabla u = \gamma^{\frac{1}{n}}\nabla \sigma^{\frac{1}{n}} + \sigma^{\frac{1}{n}}\nabla \gamma^{\frac{1}{n}},
			\end{equation*}
			taking the $L^2$-norm and using Proposition  \ref{max-principle} we obtain 
			\begin{align*}
				\int_\Omega |\nabla u|^2 d\x &\leq \int_\Omega |\gamma^{\frac{1}{n}}\nabla \sigma^{\frac{1}{n}}|^2 d\x + \int_\Omega |\sigma^{\frac{1}{n}}\nabla \gamma^{\frac{1}{n}}|^2 d\x \\
				&\leq (\Gamma_{+})^{\frac{2}{n}}\int_\Omega |\nabla \sigma^{\frac{1}{n}}|^2 d\x + (\Gamma_{+})^{\frac{2(n-m)}{nm}}\int_\Omega |\nabla \gamma^{\frac{1}{n}}|^2 d\x \\
				&=(\Gamma_{+})^{\frac{2}{n}}\int_\Omega \Big|\frac{1}{n}\sigma^{\frac{1}{n}-1}\nabla \sigma \Big|^2 d\x + (\Gamma_{+})^{\frac{2(n-m)}{nm}}\int_\Omega \Big|\frac{1}{n}\gamma^{\frac{1}{n}-1}\nabla \gamma\Big|^2 d\x \\
				&\leq \Big(\frac{1}{n}\Big)^2 (\Gamma_{+})^{\frac{4(n-m)(1-n)}{mn^2}}\int_\Omega |\nabla \sigma|^2 d\x + \Big(\frac{1}{n}\Big)^2 (\Gamma_{+})^{\frac{4(n-m)(1-n)}{mn^2}}\int_\Omega |\nabla \gamma|^2 d\x \\
				&\leq \hat{C_1}\int_\Omega |\nabla \sigma|^2 d\x +\hat{C_2}\int_\Omega |\nabla \gamma|^2 d\x, 
			\end{align*}
			where the constants $\hat{C_1}, \hat{C_2}$ depend on $m,n, \Gamma_{-}$ and $\Gamma_{+}$. By Lemma \ref{asympt-sigma}, we obtain 
			\begin{equation*}
				\lim_{t\rightarrow\infty}\int_\Omega |\nabla u|^2 d\x = 0.
			\end{equation*}
\end{proof}
%	\end{corollary}
	We now prove Theorem \ref{main-theorem}.
	\begin{proof}
    First, we note that if we normalize the initial data so that $\frac{1}{|\Omega|}\int_\Omega u_0(\textbf{y}) = 1 $, then from the conservation law \eqref{conslaw} we have $\frac{1}{|\Omega|}\int_\Omega u(\textbf{y},t) d\textbf{y} = 1$.
		Applying the Poincar\'e inequality, there is a constant $C >0$ such that  
		\begin{equation}
			\Big\|u(\textbf{x},t)-\fint_\Omega u(\textbf{y},t)d\textbf{y}\Big\|_{L^2(\Omega)} \leq C \|\nabla u\|_{L^2(\Omega)}
		\end{equation}
and so by Corollary \ref{lemma-conv-ugamma} we have 
		\begin{equation*}
			\lim_{t\rightarrow \infty}\|u(\textbf{x},t)-1\|_{L^2(\Omega)} =0.    
		\end{equation*}
        Now, for the convergence of $\gamma$, we multiply $\eqref{pmi}_2$ by $e^t$ to obtain 
        \begin{equation}
            \|\gamma(\x,t) - 1\|_{L^2(\Omega)} \leq e^{-t} \|\gamma_0-1\|_{L^2(\Omega)} + \int_0^t e^{-(t-s)} \|u^m-1\|_{L^2(\Omega)} ds,
        \end{equation}
        where, thanks to Proposition \ref{max-principle}, the right-hand side converges to $0$ as $t\rightarrow \infty$.
	\end{proof}

%\vfil\eject
%  section self-similar localizing profiles
%%%%%%%%%%%%%%%%%%%%%%%%%%%%%%%%%%%%%%%%%%%%%%%%%%%%%%%%%%%%%%%%
%%%%%%%%%%%%%%%%%%%%%%%%%%%%%%%%%%%%%%%%%%%%%%%%%%%%%%%%%%%%

	\section{Self-similar localizing solutions for \texorpdfstring{$n<m$}{n<m}}\label{sec:localization}
	In the following sections we consider the system \eqref{pmi} on $\Omega = \R^d$  in the parameter range $m>n>0$, 
	and construct solutions that exhibit localization. 
	We consider the evolution of radially symmetric solutions of \eqref{pmi}, $u = u(\rho, t)$, $\gamma = \gamma (\rho, t)$ with $\rho = |x|$,
	which satisfy the equations 
	\begin{equation}
		\tag{\tcb{$P_{\rho,t}$}}
		\begin{cases}
			\partial_t u = \partial_{\rho\rho} (\frac{1}{\gamma}u^n) + {d-1 \over \rho} \partial_\rho (\frac{1}{\gamma}u^n),\\
			\partial_t \gamma = -\gamma +u^m.
		\end{cases}
		\label{eq:radial}
	\end{equation}
 Figure~\ref{fig:Localexample} illustrates a numerical simulation of \eqref{eq:radial} for $d=1$, obtained by using a fully implicit finite volume scheme combined with a Newton iteration. The solution $(u,\gamma)$ develops a localized profile, concentrating at a single point. 
 When the initial data is even, the concentration occurs at $\rho = 0$. 
 \begin{figure}[H]
		\centering
		
		\begin{subfigure}{0.35\textwidth}
			\centering
			\includegraphics[width=\textwidth]{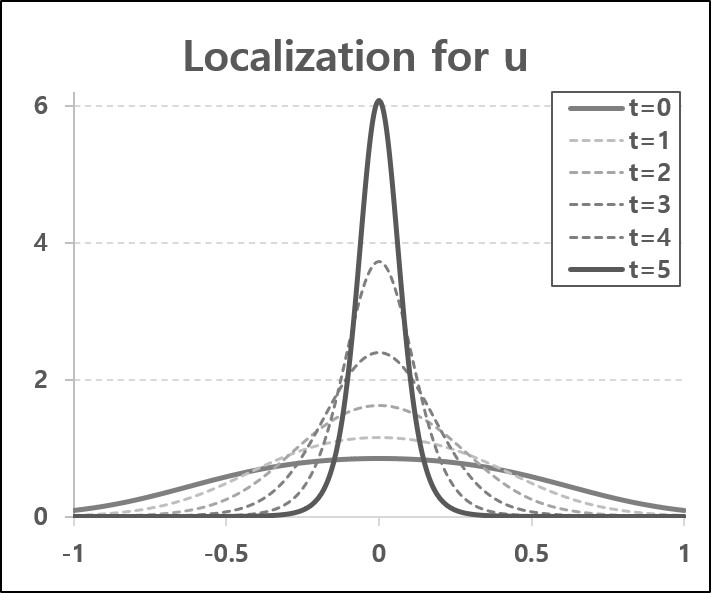}
			\caption{$u(\rho,t)$}
		\end{subfigure}
		\hspace{5mm}
		\begin{subfigure}{0.35\textwidth}
			\centering
			\includegraphics[width=\textwidth]{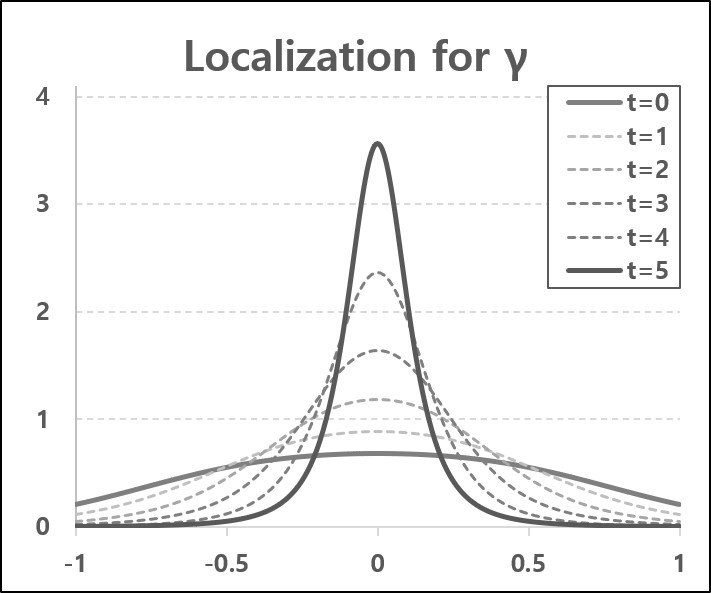}
			\caption{$\gamma(\rho,t)$}
		\end{subfigure}
		
		\caption{Numerical simulation of the equation \eqref{eq:radial} shows localization at $\rho=0$. The parameters used for the simulation are $d=1$, $n=0.5$, $m=0.9$.}
		\label{fig:Localexample}
\end{figure}

 \subsection{Self-similar localizing solutions}
To study this response, we introduce the {\it ansatz} of a self-similar localizing solution of the form	
\begin{equation}\label{eq:firstans}
u(\rho,t)=e^{a\lambda t}\bar U(\xi),
\qquad
\gamma(\rho,t)=e^{b\lambda t}\bar\Gamma(\xi),
\qquad
\xi=e^{\lambda t}\rho,
\end{equation}
with $\lambda>0$. This ansatz is motivated by the the scaling invariance property \eqref{scin}, with
constants $a$, $b$ determined by
\begin{equation}\label{coeff}
		a := a^{m,n} = {2 \over 1+m-n}, \quad b := b^{m,n} = {2m \over 1+m-n} \, .
	\end{equation}
The functions $(\bar U,\bar\Gamma)$ determine the profile of the solution.

\medskip
We define the notion of localizing solution, as follows:
\begin{definition}\label{def:localization}
		A positive smooth solution $(u,\gamma)$ of the diffusion-relaxation system \eqref{eq:radial} is called a localizing solution if it 
		satisfies the ansatz
		\begin{equation}\label{selfsimloc}
			u(\rho,t) = A(t) \ \bar{U}\big (f(t)  \ \rho\big), \quad \gamma(\rho,t) = B(t) \ \bar{\Gamma}\big ( f(t) \ \rho \big)
		\end{equation} 
		where $f(t)$,  $A(t)$, $B(t)$ are increasing functions on $[0,\infty)$, $f(0) = A(0) = B(0) = 1$, 
		and the solution satisfies 
		\begin{enumerate}
			\item when $\rho=0$, 
			\begin{align*}
				&\lim\limits_{t\to \infty} u(0,t) = +\infty, \quad \lim\limits_{t\to \infty} \gamma(0,t) = +\infty,
			\end{align*}
			\item when $\rho\ne 0$,
			\begin{align*}
				&\lim\limits_{t\to \infty} u(\rho,t) = 0, \quad \lim\limits_{t\to \infty} \gamma(\rho,t) = 0.
			\end{align*}
		\end{enumerate}
	\end{definition}
 
 \noindent
    A localizing solution in the sense of  Definiton~\ref{def:localization} is a special solution of the system~\eqref{eq:radial} with initial data 
    $u_0(\rho) = \bar{U}\left(\rho\right), \; \gamma_0(\rho) = \bar{\Gamma}\left(  \rho\right)$.
        As time increases, the self-similar format of the function \eqref{selfsimloc}  tends to concentrate information around the origin $\rho =0$.
In the case of \eqref{eq:firstans}, we have selected $f(t) = e^{\lambda t}$, $A(t) = e^{a\lambda t}$, and $B(t) = e^{b\lambda t}$ with $a, b, \lambda > 0$. 
Note that \eqref{eq:firstans} will fit into the definition of localizing solution provided that the profile equations can be selected to satisfy
$(\bar U,\bar\Gamma) (\xi) \to 0$ as $\xi \to \infty$.
    
\begin{remark}
 The reader should note that condition (2) of Definition \ref{def:localization} might be too restrictive and one could replace it with the relaxed condition
 \begin{eqnarray*}
 &u(\rho, t) \sim {\hat A} (t)  \quad \mbox{as $\rho \to \infty$}&
 \\
 &\gamma (\rho, t) \sim {\hat B}(t) \quad \mbox{as $\rho \to \infty$}&
 \\
 &\mbox{ with  ${\hat A} (t) = o (A(t))$, ${\hat B}(t) = o(B(t))$, as $t \to \infty$}.&
 \end{eqnarray*}
 and still have a localizing response.
 \end{remark}

 \subsection{Derivation of the problem determining the profiles}

Under the {\it ansatz} \eqref{eq:firstans}, the localization problem is reduced to finding the profiles $(\bar U,\bar\Gamma)$. 
Since $\xi=e^{\lambda t}\rho\to\infty$ as $t\to\infty$, if the profile decays as $\xi\to\infty$, then the solution decays away from the origin, while the exponential prefactors produce growth at $\rho=0$. Indeed, localization is encoded in the boundary behavior of the profile,
and the existence of localizing solutions is reduced to a boundary value problem for the self-similar profile.  
The details of these considerations are captured below.

We proceed to devise the problem that will determine the profile $(\bar{U}(\xi), \bar{\Gamma}(\xi))$.  Introducing
the ansatz \eqref{eq:firstans} to the system \eqref{eq:radial}, we derive the system of ordinary differential equations
\begin{align} \label{nonautoODE}
		&\begin{aligned}
			\frac{2}{1+m-n} \lambda \bar{U} + \lambda \xi \bar{U}' &= \big(\frac{1}{\bar{\Gamma}} \ \bar{U}^n \big)'' + {d-1 \over \xi} \big(\frac{1}{\bar{\Gamma}} \ \bar{U}^n \big)', \\ \frac{2m}{1+m-n} \lambda \bar{\Gamma}+\lambda \xi \bar{\Gamma}' &= -\bar{\Gamma} + \bar{U}^m.
		\end{aligned} \tag{$PS$}
\end{align}
The prime notation denotes differentiation with respect to $\xi$, $f' = {d f \over d\xi} $. 

\begin{remark}\label{rem:scalingprop}
It can be easily checked that the system \eqref{nonautoODE} is itself invariant under the scaling
$$
\bar {U_c} (\xi)  = c^a \bar{U}(c\xi) \, , \quad  \bar \Gamma_c (\xi) = c^b \bar{\Gamma}(c\xi),
$$
where the constants $a$, $b$ are again selected by \eqref{coeff}. Furthermore, if there is a self-similar solution of the system \eqref{nonautoODE}
under this scaling, then this solution has the simple form
$$
\bar{U}_{ss} (\xi) = U(1) \xi^{-a}, \quad \bar{\Gamma}_{ss} (\xi) = \Gamma(1) \xi^{-b}.
$$
The reader can check by a direct computation that such self-similar solutions do not exist. On the other hand their format motivates
a transformation that will be successful to de-singularize the problem in the following section.
\end{remark}

The system \eqref{nonautoODE} is not autonomous and it has the property that it is invariant under the change of variable 
$\xi \to (-\xi)$.  Since we want a smooth and positive solution on the real line, we impose the initial conditions
	\begin{equation} \label{eq:BD}
		\bar{U}(0) > 0, \quad \bar{\Gamma}(0)>0, \quad \bar{U}^\prime (0) = 0, \quad \bar{\Gamma}^\prime (0)=0 ,
	\end{equation}
solve the problem in the interval $(0, \infty)$ and reflect the solution so that $\bar U (\xi) = \bar U (-\xi)$, $\bar \Gamma (\xi) = \bar \Gamma (-\xi)$,
for $\xi \in (-\infty, 0)$. The conditions $U^\prime(0) = \Gamma^\prime (0)=0$ are imposed to secure a smooth profile.
The values $\bar{U}(0)$ and $\bar{\Gamma}(0)$ determine the coefficients of the growth rate at $\xi = 0$:
	\begin{equation} \label{def:U0}
		u(0,t) = \bar{U}(0) e^{a\lambda t}, \quad \gamma(0,t) = \bar{\Gamma}(0) e^{b\lambda t}.
	\end{equation}

A consequence of the scaling invariance of the system \eqref{nonautoODE} is that we may assume, without loss of generality,
that the initial datum $\Gamma (0) = 1$.
	By substituting $\xi = 0$ in the second equation of \eqref{nonautoODE} and using $\bar{\Gamma}(0) =1$, we have 
	\begin{equation*}
		 \frac{2m}{1+m-n} \lambda = -1 + \bar{U}^m(0).
	\end{equation*}

{\bf The profile equations}. The problem determines the profile $(\bar U, \bar \Gamma)$ thus consists of solving the system \eqref{nonautoODE}
for $\xi \in (0, \infty)$ subject to the initial condition
\begin{equation}\label{icPS}
		\bar{U}(0) = U_0 \, , \quad \bar{\Gamma}(0) = 1 \, , \quad \bar{U}^\prime (0) = 0, \quad \bar{\Gamma}^\prime (0)=0.
\end{equation}
We will also impose the condition
\begin{equation}\label{decay}
\bar{U} (\xi) \to 0 \, \quad \bar{\Gamma} (\xi) \to 0 \, , \quad \mbox{as $\xi \to \infty$}
\end{equation}
so that $(\bar U, \bar \Gamma)$ is a localizing solution. The resulting problem determining the profiles is a 
boundary value problem for the non-autonomous system \eqref{nonautoODE}.

%\bigskip
%\bigskip
%
%\tcr{\bf I Edited up to here}
%
%The first property is regarding the behavior around $\xi =0$
%
%\tcr{Please develop this argument here}
%
%The second property characterizes the expected decay rate for large $\xi$. It proceeds via a heuristic 
%argument outlined in appendix \ref{sec:heu}.
%
%\tcr{Please write the consequences of the argument here}
%
%
%\tcr{ Also we need to list other properties that are listed in this section.}
%
%
%
%
%\bigskip
%\bigskip

%\color{blue}

Of course, it is not {\it a-priori} clear that this boundary value problem admits a solution. We first identify several necessary properties of the localizing profiles under consideration. These concern their behavior near the origin, their asymptotic decay as $\xi\to\infty$, and the resulting restrictions on the parameters.

\medskip
\noindent
{\bf Behavior near the origin.}
By the normalization $\bar{\Gamma}(0)=1$ and the second equation of \eqref{nonautoODE}, the value $U_0=\bar{U}(0)$ is related to the growth rate $\lambda$ through
\begin{equation}\label{eq:lU0}
	U_0^m=1+b\lambda=1+\frac{2m}{1+m-n}\lambda.
\end{equation}
In particular, $\lambda>0$ implies $U_0>1$. Smoothness and radial symmetry of the profiles require
\[
\bar{U}'(0)=\bar{\Gamma}'(0)=0.
\]
Consequently, the profiles admit the local expansions
\begin{equation}\label{eq:origin-expansion}
	\bar{U}(\xi)=U_0+\frac{1}{2}\bar{U}''(0)\xi^2+o(\xi^2), \qquad \bar{\Gamma}(\xi)=1+\frac{1}{2}\bar{\Gamma}''(0)\xi^2+o(\xi^2) \qquad\text{as }\xi\to0.
\end{equation}
We seek profiles that decrease immediately away from the origin. Since their first derivatives vanish at $\xi=0$, such profiles should be locally concave at the origin; in particular, we require
\[
\bar{U}''(0)<0, \qquad \bar{\Gamma}''(0)<0.
\]
In contrast, an important feature of the transformed quantity
\[
\bar{\Sigma}(\xi):=\frac{\bar{U}^n(\xi)}{\bar{\Gamma}(\xi)}
\]
is that it is locally convex near the origin. Indeed, $\bar{\Sigma}'(0)=0$, and taking the limit $\xi\to0$ in the first equation of \eqref{nonautoODE} gives
\begin{equation}\label{eq:origin-sigma}
	d\,\bar{\Sigma}''(0)=a\lambda U_0.
\end{equation}
Since $a$, $\lambda$, and $U_0$ are positive, it follows that $\bar{\Sigma}''(0)>0$. Moreover,
\[
\bar{\Sigma}''(0)=nU_0^{n-1}\bar{U}''(0)-U_0^n\bar{\Gamma}''(0).
\]
Expanding the second equation of \eqref{nonautoODE} to second order also gives
\begin{equation}\label{eq:origin-second}
	\bigl(1+(b+2)\lambda\bigr)\bar{\Gamma}''(0)=mU_0^{m-1}\bar{U}''(0).
\end{equation}
Thus, once $U_0$, and hence $\lambda$, is fixed, equations \eqref{eq:origin-sigma} and \eqref{eq:origin-second} determine the quadratic behavior of the profiles near the origin. These local expansions will later select the appropriate direction of the orbit in the autonomous formulation. A direct calculation gives
\begin{align*}
	\bar{U}''(0)&=-\frac{a\lambda(U_0^m+2\lambda)}{d\bigl((m-n)U_0^m-2n\lambda\bigr)}U_0^{2-n},\\
	\bar{\Gamma}''(0)&=-\frac{ma\lambda}{d\bigl((m-n)U_0^m-2n\lambda\bigr)}U_0^{m+1-n}.
\end{align*}
Therefore, the required local concavity of $\bar{U}$ and $\bar{\Gamma}$ imposes the condition
\begin{equation}\label{eq:assumpconcave}
	(m-n)U_0^m-2n\lambda>0.
\end{equation}

\medskip
\noindent
{\bf Expected behavior at infinity.}
To obtain localization away from the origin, the profiles must decay as $\xi\to\infty$. We seek profiles with polynomial asymptotic behavior
\begin{equation}\label{eq:BD2}
	\bar{U}(\xi)=C_1\xi^{-\mu_1}+o\left(\xi^{-\mu_1}\right), \qquad \bar{\Gamma}(\xi)=C_2\xi^{-\mu_2}+o\left(\xi^{-\mu_2}\right) \qquad\text{as }\xi\to\infty,
\end{equation}
where $C_1,C_2,\mu_1$, and $\mu_2$ are positive constants. Substituting \eqref{eq:BD2} into the second equation of \eqref{nonautoODE}, we observe that the terms involving $\bar{\Gamma}$ are of order $\xi^{-\mu_2}$, whereas $\bar{U}^m$ is of order $\xi^{-m\mu_1}$. By assuming the relaxation equation is balanced, these terms have the same asymptotic order and it gives
\[
\mu_2=m\mu_1.
\]
The detailed calculation is presented in Appendix~\ref{sec:heu}.
For $d=1$, the balance in the first profile equation then determines
\begin{equation}\label{eq:decay-exponents}
	\mu_1=\frac{1}{m-n}, \qquad \mu_2=\frac{m}{m-n}.
\end{equation}
The leading-order coefficients satisfy
\begin{equation}\label{eq:decay-coefficients}
	\frac{C_1^m}{C_2}=1-\frac{m\lambda(1-m+n)}{(m-n)(1+m-n)}.
\end{equation}
For a fixed $\rho\neq0$, substituting \eqref{eq:BD2} into the self-similar ansatz \eqref{eq:firstans} gives
\begin{align*}
	u(\rho,t)&=C_1\rho^{-\mu_1}e^{(a-\mu_1)\lambda t}+o\left(e^{(a-\mu_1)\lambda t}\right),\\
	\gamma(\rho,t)&=C_2\rho^{-\mu_2}e^{(b-\mu_2)\lambda t}+o\left(e^{(b-\mu_2)\lambda t}\right).
\end{align*}
Therefore, decay away from the origin requires
\begin{equation}\label{eq:localization-exponents}
	\mu_1>a, \qquad \mu_2>b.
\end{equation}
For the exponents in \eqref{eq:decay-exponents}, these two inequalities are equivalent to
\[
m<n+1.
\]
Since $C_1$ and $C_2$ are positive, the right-hand side of \eqref{eq:decay-coefficients} must also be positive. This leads to the restriction
\begin{equation}\label{eq:lambda-restriction}
	0<\lambda<\frac{(m-n)(1+m-n)}{m(1-m+n)}.
\end{equation}
Consequently, the parameter range
\[
0<n<m<n+1, \qquad 0<\lambda<\frac{(m-n)(1+m-n)}{m(1-m+n)}
\]
is naturally associated with the construction of localizing profiles.

Finding a localizing solution of \eqref{eq:radial} is thus reduced to finding a positive smooth solution of the non-autonomous system \eqref{nonautoODE} that satisfies the conditions at the origin \eqref{icPS} and the asymptotic behavior \eqref{eq:BD2}. The following proposition explains how these profile conditions imply localization of the corresponding solution of the original system.

\color{black}

	\begin{proposition} \label{prop:growth}
		Let $(\bar{U}(\xi), \bar{\Gamma}(\xi))$ be a positive smooth solution of the system \eqref{nonautoODE} with a boundary condition \eqref{eq:BD}. Let $\lambda>0$ be chosen to satisfy the polynomial decay \eqref{eq:BD2} with $\mu_1 > a^{m,n}$ and $\mu_2 > b^{m,n}$. Then $(u,\gamma)$ is a localizing solution at $\rho=0$ of the system \eqref{eq:radial}.
	\end{proposition}
	\begin{proof}
		From the ansatz \eqref{eq:firstans}, $(u,\gamma)$ is a positive smooth solution. When $\rho=0$, we have 
		\begin{equation*}
			u(0,t) = \bar{U}(0) e^{a \lambda t}, \quad \gamma(0,t) = e^{b \lambda t}.
		\end{equation*}
		Since $m >n >0$, the parameters $a,b, \lambda$ are positive. Moreover, $u(0,t) \to +\infty$ and $\gamma(0,t) \to +\infty$ as $t \to \infty$.
		From the polynomial decay \eqref{eq:BD2}, as $t \to \infty$ with $\rho\ne 0$, $(u,\gamma)$ satisfy
		\begin{equation*}
			u(\rho,t) = C_1 e^{(a-\mu_1) \lambda t} \ \rho^{-\mu_1} + o(e^{(a-\mu_1) \lambda t}), \quad \gamma(\rho,t) = C_2 e^{(b-\mu_2) \lambda t} \ \rho^{-\mu_2} + o(e^{(b-\mu_2) \lambda t}).
		\end{equation*}
		Since $a-\mu_1<0$ and $b-\mu_2<0$, we have $u(\rho,t) \to 0$ and $\gamma(\rho,t) \to 0$ as $t \to \infty$.
	\end{proof}

%	\vfil\eject
		
\section{Reduction to a heteroclinic orbit for an autonomous system}\label{sec:reduction}

In this section we reduce the problem to that of determining a heteroclinic orbit for an autonomous system.
The procedure consists of two steps: (i) first, we desingularise the problem, (ii) second, we transform it to an 
appropriate autonomous system in new variables. The latter is studied in the following section and we prove an existence 
theorem for the heteroclinic orbit that is stated at the end of this section.

\subsection{Desingularization of the system, reduction to an autonomous problem}

	To reduce the order of \eqref{nonautoODE}, we define $\bar{\Sigma} := \frac{\bar{U}^n}{\bar{V}}$ and $\bar{W}:= \bar{\Sigma}'$. The system \eqref{nonautoODE} can be written as
	\begin{align} \label{nonautoODE2}
		&\begin{aligned}
			\frac{2}{1+m-n}\lambda \bar{U} + \lambda \xi \bar{U}' &=  \bar{W}'+{d-1 \over \xi} \bar{W}, \\ \frac{2m}{1+m-n} \lambda \bar{\Gamma}+\lambda \xi \bar{\Gamma}' &= -\bar{\Gamma} + \bar{U}^m, \\ \bar{\Sigma}' &= \bar{W}.
		\end{aligned} \tag{$P_{\xi}'$}
	\end{align}

Next, we introduce a change of variables for time augmentation of the system \eqref{nonautoODE2}.
	\begin{equation}\label{eq:secondans}
		\bar{U}(\xi) = \xi^{-a} U( \ln \xi), \ \bar{\Gamma}(\xi) = \xi^{-b} \Gamma( \ln \xi), \ \bar{W}(\xi) = \xi^{-\omega} W( \ln \xi)
	\end{equation}
where 
	\begin{equation*}
		a = a^{m,n} = \frac{2}{1+m-n}, \quad b = b^{m,n} = \frac{2m}{1+m-n}, \quad \omega = a^{m,n}-1 \, .
	\end{equation*}
The transformation is motivated by the scaling properties observed in Remark \ref{rem:scalingprop}.
Since $\Sigma = U^n / \Gamma$, we also have
	\[ \bar{\Sigma}(\xi) = \frac{1}{\bar{\Gamma}} \bar{U}^n = \xi^{b - na} \Sigma( \ln \xi) = \xi^{2\frac{m-n}{1+ m - n}} \Sigma( \ln \xi). \]
Then we define a new independent variable as $$\eta = \ln \xi = \ln |x| + \lambda t $$ with the upper dot notation $\dot{f} = {df \over d\eta}$. Substituting these to the system \eqref{nonautoODE2},
	\begin{equation}\label{autoODE2}
		\begin{aligned}
			\dot{W} - \lambda \dot{U} &= \left(\frac{2}{1+m-n} -d \right) \ W, \\ 
			\lambda \dot{\Gamma} &= -\Gamma + U^m, \\
			\dot{\Sigma} &= -\frac{2(m-n)}{1+m-n} \ \Sigma + W.
		\end{aligned} 
	\end{equation}
		The original function $u$ and $\gamma$ can be recovered as 
	\begin{equation}\label{eq:recovered}
		u(\rho,t) = \rho^{-a} U( \ln \rho + \lambda t ), \ \gamma(\rho,t) = \rho^{-b} \Gamma( \ln \rho + \lambda t ).
	\end{equation}
	
	%\tcr{If $m = n+1$, we have $a^{m,n} = 1$. Then, the first equation of the system \eqref{autoODE2} can be integrated, which gives a planar system on the manifold $W-\lambda U = C$. Analysis for the planar system will be discussed in Appendix~\ref{sec:n+1}. In this section, we assume that $m \ne n+1$.}
	
	By defining $Z = W-\lambda U$, we can rearrange \eqref{autoODE2} to derive the first order autonomous ODE system.
	\begin{equation}\label{autoODE}
		\begin{aligned}
			\dot{Z} &= \left(\frac{2}{1+m-n} -d \right) \left(Z + \lambda U \right), \\ 
			\dot{\Sigma} &= -\frac{2(m-n)}{1+m-n} \Sigma + Z + \lambda U, \\
			\lambda \dot{\Gamma} &= -\Gamma + U^m.
		\end{aligned} \tag{$P_{\eta}$}
	\end{equation}
	Where $U = (\Sigma \Gamma)^{1\over n}$. From the relation $\xi = e^\eta$, the trajectory of \eqref{nonautoODE2} for $\xi\in [0,+\infty)$ corresponds to the trajectory of \eqref{autoODE} for $\eta \in (-\infty,+\infty)$. 
	
%\color{blue}
However, $\Sigma(\eta)$ diverges as $\eta \to -\infty$, since
\[
\lim_{\eta\to -\infty} \Sigma(\eta)
= \lim_{\xi\to 0} \xi^{na-b}\bar{\Sigma}(\xi)
= \bar{U}^n(0)\lim_{\xi\to 0}\xi^{-(m-n)a}
= \infty.
\]
Consequently, the analysis of \eqref{autoODE} would involve an unbounded
trajectory. To avoid the resulting technical difficulties, we introduce a
nonlinear transformation to a $pqr$-system in which the corresponding
trajectory is expected to remain bounded. The choice of such a transformation
is not immediate. A useful starting point is to rewrite \eqref{autoODE} in
terms of logarithmic derivatives:
\begin{equation}\label{autoODElog}
	\begin{aligned}
		\frac{d}{d\eta}\left(\ln Z\right)
		&= \left(\frac{2}{1+m-n}-d\right)
		\left(1+\lambda\frac{U}{Z}\right), \\
		\frac{d}{d\eta}\left(\ln\Sigma\right)
		&= -\frac{2(m-n)}{1+m-n}
		+\frac{Z}{\Sigma}+\lambda\frac{U}{\Sigma}, \\
		\lambda\frac{d}{d\eta}\left(\ln\Gamma\right)
		&= -1+\frac{U^m}{\Gamma}.
	\end{aligned}
	\tag{$\log P_{\eta}$}
\end{equation}
As $\eta\to+\infty$, the quantities appearing on the right-hand side of
\eqref{autoODElog} are expected to approach constants. More precisely, we
expect that
\[
\frac{U}{Z}, \qquad
\frac{Z}{\Sigma}+\lambda\frac{U}{\Sigma},
\qquad\text{and}\qquad
\frac{U^m}{\Gamma}
\]
converge to constants, and the same is therefore expected for suitable
combinations of these quantities. Among the possible choices suggested by
this observation, we select one in the following subsection.

\subsection{Transformation to an equivalent autonomous system}

In this subsection, we introduce an equivalent autonomous system for which
the corresponding trajectory is expected to remain bounded as $\eta$ varies
from $-\infty$ to $+\infty$. The new system is formulated in terms of three
functions $p$, $q$, and $r$ of the variable $\eta$, defined through $Z$,
$\Sigma$, and $\Gamma$. We seek a transformation satisfying the following
two requirements:
\begin{itemize}
	\item In terms of the variable $\xi$, the quantity
	$(p,q,r)(\eta(\xi))$ remains bounded as $\xi\to 0^+$.
	\item In terms of the variable $\eta$, the expected asymptotic behavior
	as $\eta\to+\infty$, suggested by \eqref{autoODElog}, is also bounded.
\end{itemize}

\color{black}

%\tcr{We need to include some motivation here} % comment by HYK : See the above blue colored part.
%By heuristic calculations from the system \eqref{nonautoODE2}, the following three pairs of quantities are expected to share the same asymptotic leading order as $\xi \to 0$. 
%	\begin{equation*}
%		\bar{U} - \xi^{-1} \bar{W} = O(\xi), \quad \bar{\Gamma} - \bar{U}^m = O(\xi), \quad \bar{W} - \xi^{-1} \bar{\Sigma} = O(\xi)
%	\end{equation*}

	This suggests introducing the new variables $(p,q,r)(\eta)$;
	\begin{align*}
		&p:= \frac{\xi^2 \bar{\Gamma}^{1\over m}}{\bar{\Sigma}} = \frac{\Gamma^{1\over m}}{\Sigma}, \\ & q := \frac{\xi \bar{W} - \lambda \xi^2 \bar{U}}{\bar{\Sigma}} = \frac{Z}{\Sigma}, \\ & r := \frac{\bar{U}}{\bar{\Gamma}^{1\over m}} = \frac{U}{\Gamma^{1\over m}}.
	\end{align*}
	Using those new variables, substitution into \eqref{autoODE} gives the following ODE system;
	\begin{equation}\label{eq:pqr}
		\begin{aligned}
			\dot{p} &= -p \left( q+\lambda pr - \frac{1}{m\lambda} r^m -\frac{2(m-n)}{1+m-n} + {1\over m \lambda} \right),  \\ 
			\dot{q} &= q \ (2-d-q-\lambda pr) +  \left(\frac{2}{1+m-n} -d \right) \lambda pr ,\\
			n \ \dot{r} &= r\left(q+\lambda pr + \frac{m-n}{m\lambda} r^m -\frac{2(m-n)}{1+m-n} - {m-n\over m\lambda} \right). 
		\end{aligned} \tag{$P_{pqr}$}
	\end{equation}
	
	The transformation from $(Z, \Sigma, \Gamma)$ to $(p,q,r)$ can be inverted as
	\begin{equation} \label{eq:transform}
		\begin{aligned}
			& U = \left(pr^{1+m} \right)^{1\over 1+m-n}, \quad \Gamma = \left(pr^{n} \right)^{m\over 1+m-n}, \quad \Sigma = \left(p^{n-m}r^{n} \right)^{1 \over 1+m-n}, \\ & W = q \left(p^{n-m}r^{n} \right)^{1 \over 1+m-n} + \lambda \left(pr^{1+m} \right)^{1\over 1+m-n} = \Sigma (q+\lambda pr)
		\end{aligned}
	\end{equation}
	When $n = m+1$, the system becomes $pr^{n} = 1$ and the transform from $Z \Sigma \Gamma$--system to $pqr$--system is not invertible. Since we are in the parameter domain of $m>n$, this transform is always invertible.

 \subsection{Statement of results for dimension $d=1$}.

The problem of construction of the profile has been transformed to determining a heteroclinic orbit for  the system \eqref{eq:pqr}.
We now focus on the latter problem from a perspective of dynamical systems. We follow an approach devised (for shear bands problems) in 	
\cite{LeeTzavaras2017,LeeKatsaounisTzavaras2019}. 
The idea is to view the system \eqref{eq:pqr} as a small perturbation as $n \to 0$ of an associated problem and use the geometric theory
of singular perturbations to establish the existence of a heteroclinic orbit as a geometric object invariant under small perturbations. 
This is done in Sections  \ref{sec:asymptotic} and \ref{sec:gspt}. The calculation of equilibria and general properties are valid for any dimension, but the 
application of the geometric theory is only valid for dimension $d=1$. We state the final result.

	\begin{theorem}[Existence of a localizing solution] \label{thm:exist}
		Assume $d=1$ and $$0<m<1, \quad 0<\lambda < {1+m \over 1-m}. $$ Then, there exists a small $n_0>0$ satisfying 
		\begin{equation} \label{eq:assumpn0}
			n_0 < m, \quad \lambda < {(m-n_0)(1+m-n_0) \over m(1-m+n_0)}, \quad (m-n_0)\left(1+\frac{2m}{1+m-n_0}\lambda \right)-2n_0\lambda>0.
		\end{equation}
		such that for $n \in (0,n_0)$, the system \eqref{eq:radial} admits a localizing solution $(u,\gamma)$ of the form \eqref{eq:firstans} which is a positive smooth radial-symmetric function and satisfies $$ u(0,0) = \left( 1+{2m \over 1+m-n} \lambda \right)^{1\over m}, \quad \gamma(0,0) = 1, \quad \partial_\rho u(0,t) = \partial_\rho \gamma(0,t) = 0. $$ 
	\end{theorem}

%\vfil\eject
%  section nonlinear stable
%%%%%%%%%%%%%%%%%%%%%%%%%%%%%%%%%%%%%%%%%%%%%%%%%%%%%%%%%%%%%%%%

	\section{Asymptotic analysis of the heteroclinic orbit for small \texorpdfstring{$n$}{n}} \label{sec:asymptotic}
	The system \eqref{eq:pqr} allows four critical points $M_i^{m,n,\lambda}$ for $i=0,1,2,3$. 
	\begin{align*}
		M_0^{m,n,\lambda} &:= \left( 0, \ 0, \ \left( 1+{2m \over 1+m-n} \lambda \right)^{1\over m} \right), \\
		M_1^{m,n,\lambda} &:= \left( 0, \ 2-d, \ \left( 1+\frac{m }{m-n} \left( d-{2 \over (1+m-n)}  \right) \lambda  \right)^{1\over m} \right),\\
		M_2^{m,n,\lambda} &:= (0, \ 0, \ 0 ),\\
		M_3^{m,n,\lambda} &:= (0, \ 2-d, \ 0).
	\end{align*}
	Note that when $d=2$, we have $M_0^{m,n,\lambda} = M_1^{m,n,\lambda}$ and $M_2^{m,n,\lambda} = M_3^{m,n,\lambda}$, so there exist only two distinct critical points.
	
	In this section, we will investigate the dynamical structure of the targeted orbit as the following steps;
	\begin{enumerate}[(i)]
		\item In the subsection~\ref{subsec:char}, we will specify a heteroclinic orbit on the $pqr$-system with boundary behavior \eqref{eq:BD} and \eqref{eq:BD2}.
		
		\item In the subsection~\ref{subsec:eigen}, we will investigate the local eigenstructure of all four critical points and their stability type on the parameter domain.
		
		\item In the subsection~\ref{subsec:asymp}, we will describe the asymptotics of the targeted heteroclinic orbit as $\eta \to \pm \infty$. 
	\end{enumerate}
	
	Before starting the analysis, we denote some frequently used quantities as alphabetical symbols. Recall that from \eqref{scin}, we have
	\begin{equation*}
		a^{m,n} = {2 \over 1+m-n}, \quad b^{m,n} = {2m \over 1+m-n},
	\end{equation*}
	and we define
	\begin{align*}
		A^{m,n,\lambda} &= \left( 1+{2m \over 1+m-n} \lambda \right)^{1\over m}, \\ B^{m,n,\lambda} &= \left( 1+\frac{m }{m-n} \left( d-{2 \over (1+m-n)}  \right) \lambda  \right)^{1\over m}.
	\end{align*}
	In the following calculation, we omit the superscript for readability. Using the above notation, the first two equilibrium points are expressed as 
	\begin{equation*}
		M_0^{m,n,\lambda} = \left( 0,0,A^{m,n,\lambda} \right), \quad M_1^{m,n,\lambda} = \left( 0,2-d,B^{m,n,\lambda} \right).
	\end{equation*} 
	
	\subsection{Characterization of a heteroclinic orbit}~\label{subsec:char}
	
	In the $pqr$-system \eqref{eq:pqr}, there might exist infinitely many heteroclinic orbits among the critical points. By interpreting the boundary conditions \eqref{eq:BD} and \eqref{eq:BD2} under the $pqr$-system \eqref{eq:pqr}, we will single out the starting and end points of the targeted heteroclinic orbit. 
	
	In Section~\ref{sec:gspt}, we will approximate the system to the singular limit $n\to0$. As a result, the dynamical equation for $r$ is reduced to the union of two-dimensional manifolds. One is the manifold $\left\{ r=0, \ p \ge 0 \right\}$ and the other is
	\begin{equation}\label{def:cm}
		S^{m,n,\lambda} := \left\{ (p,q,r) \mid q+\lambda pr + \frac{m-n}{m\lambda} r^m -\frac{2(m-n)}{1+m-n} - {m-n\over m\lambda} = 0, \ p \ge 0, \ r \ge 0 \right\}.
	\end{equation}
	Note that $p\ge 0$ and $r\ge 0$ came from the definitions of $p$ and $r$, respectively. A direct calculation using Proposition~\ref{prop:initial} shows that the starting point $M_0^{m,n,\lambda}$ is placed on the second manifold $S^{m,n,\lambda}$. Therefore, the targeted heteroclinic orbit is placed on the manifold $S^{m,n,\lambda}$, which connects $M_0^{m,n,\lambda}$ to $M_1^{m,n,\lambda}$. To make the targeted orbit well-defined, it is required to have $d=1$ and
	\begin{equation}
		1-{m(1-m+n) \over (m-n)(1+m-n)} \lambda > 0,
	\end{equation} which follow from the calculation of Appendix~\ref{sec:heu}. From now on, we consider the $1$-dimensional space domain only for the construction of the targeted heteroclinic orbit.
	
	\begin{proposition}\label{prop:initial}
		Assume $d=1$. Let $(\bar{U},\bar{\Gamma})(\xi)$ be a positive smooth solution of the ODE system \eqref{nonautoODE} with boundary condition \eqref{eq:BD}. Let $(p, q, r)(\eta)$ be the obtained orbit after the transformation process. Then, 
		\begin{equation}\label{eq:asymp1}
			e^{-2\eta} \left( \begin{bmatrix}
				p(\eta) \\ q(\eta) \\ r(\eta)
			\end{bmatrix} - M_0^{m,n,\lambda} \right) \quad \to \quad \kappa \vec{X}_{01} \qquad \mbox{ as } \eta \to -\infty
		\end{equation} for the constant $\kappa = A^{-n} $ and $\vec{X}_{01} =  \left(1, (a-1)A \lambda , -\frac{a A^{2 } \lambda^2 }{(m-n)A^m -2n\lambda} \right)$.
	\end{proposition} 
	The vector $\vec{X}_{01}$ will be realized as one of the eigenvectors of $M_0^{m,n,\lambda}$ in the following subsection.
	\begin{proof}
		By taking $\xi \to 0$ to the system \eqref{nonautoODE2} with $d=1$, we have 
		\begin{equation*} 
			\begin{aligned}
				a \lambda \bar{U}(0) &=  \bar{W}'(0), \\  \bar{U}(0) &= A , \\ \bar{\Sigma}'(0) &= \bar{W}(0)
			\end{aligned}
		\end{equation*}
		By definition of $\bar{\Sigma}(\xi)$, we have 
		\begin{equation*}
			\bar{\Sigma}(0) = \frac{\bar{U}^n(0)}{\bar{\Gamma}(0)} = A^n, \quad \bar{\Sigma}'(0) = 0 .
		\end{equation*}
		
		We start from the Taylor series expansion of solutions at $\xi = 0$. From the boundary condition \eqref{eq:BD} and the above observation at $\xi \to 0$, 
		\begin{align*}
			&\bar{U} (\xi) = A + {1\over 2}\xi^2 \bar{U}'' (0) + o(\xi^2) \\
			&\bar{\Gamma} (\xi) = 1 + {1\over 2}\xi^2 \bar{\Gamma}'' (0) + o(\xi^2) \\
			&\bar{\Sigma} (\xi) = A^n + {1\over 2}\xi^2 \bar{\Sigma}'' (0) + o(\xi^2) \\
			&\bar{W} (\xi) = a A \lambda \xi  + {1\over 2}\xi^2 \ \bar{W}''(0)  + o(\xi^2)
		\end{align*}
		By putting those series expansion to the system \eqref{nonautoODE2}, we have the following;
		\begin{align*}
			&\frac{\bar{U}''(0)}{\bar{U}(0)} = -\frac{a\lambda(A^m +2\lambda)}{ (m-n)A^m - 2n\lambda }  \ A^{1-n} \\
			&\frac{\bar{\Gamma}''(0)}{\bar{\Gamma}(0)} = \frac{mA^m}{A^m + 2\lambda} \ \frac{\bar{U}''(0)}{\bar{U}(0)} = -\frac{ma\lambda}{ (m-n)A^m - 2n\lambda }  \ A^{1+m-n} \\ 
			&\bar{\Sigma}''(0) = \bar{W}'(0) = a A \lambda \\
			&\bar{W}''(0) = 0
		\end{align*}
		Now, we calculate the Taylor expansion of $p(\ln \xi)$ at $\xi = 0$.
		\begin{eqnarray*}
			p(\ln \xi) &=& \frac{\xi^2 \bar{\Gamma}^{1\over m} }{\bar{\Sigma}} 
			= \frac{\xi^2\left( 1 + {1\over 2}\xi^2 \bar{\Gamma}'' (0) + o(\xi^2) \right)^{1\over m} }{A^n + {1\over 2}\xi^2 \bar{\Sigma}'' (0)  + o(\xi^2)} \\ &=& \frac{ \xi^2 +o(\xi^2) }{A^n + o(1)} = A^{-n} \xi^2  + o(\xi^2).
		\end{eqnarray*}
		For the Taylor expansion of $q(\ln \xi)$ at $\xi = 0$, 
		\begin{eqnarray*}
			q(\ln \xi) &=& \frac{\xi \bar{W} - \lambda\xi^2 \bar{U}}{\bar{\Sigma}} = \frac{aA\lambda \xi^2  - A \lambda \xi^2 + o(\xi^2)}{A^n + {1\over 2}\xi^2 \bar{\Sigma}''(0) + o(\xi^2)} \\ 
			&=& (a-1) A^{1-n} \lambda \xi^2 + o(\xi^2).
		\end{eqnarray*}
		Also for the Taylor expansion of $r(\ln \xi)$ at $\xi = 0$, 
		\begin{eqnarray*}
			r(\ln \xi) &=& \frac{\bar{U}}{\bar{\Gamma}^{1\over m}} = \frac{A + {1\over 2}\xi^2 \bar{U}'' (0) + o(\xi^2)}{\left(1 + {1\over 2}\xi^2 \bar{\Gamma}'' (0) + o(\xi^2)\right)^{1\over m}} \\ &=& A \left( 1+ {1\over 2}\xi^2 {\bar{U}'' (0) \over \bar{U}(0)} + o(\xi^2) \right)  \left(1 - {1\over 2}\xi^2 {\bar{\Gamma}'' (0) \over m\bar{\Gamma}(0)} + o(\xi^2)\right) \\
			&=& A \left( 1+ {1\over 2} \xi^2 \left( {\bar{U}'' (0) \over \bar{U}(0)} -{\bar{\Gamma}'' (0) \over m\bar{\Gamma}(0)} \right) \right) + o(\xi^2) \\
			&=& A \left( 1+ \xi^2 \frac{\lambda}{A^m+2\lambda} \ {\bar{U}'' (0) \over \bar{U}(0)} \right) + o(\xi^2).
		\end{eqnarray*}
		And we have
		\begin{equation*}
			r(\ln \xi) - A = - \xi^2 \frac{a\lambda^2}{ (m-n)A^m - 2n\lambda }  \ A^{2-n} + o(\xi^2).
		\end{equation*}
		By combining all terms together and applying $\eta = \ln \xi$,
		\begin{equation*}
			e^{-2\eta} \left( \begin{bmatrix}
				p(\eta) \\ q(\eta) \\ r(\eta)
			\end{bmatrix} - \begin{bmatrix}
				0 \\ 0 \\ A
			\end{bmatrix} \right) = \kappa \vec{X}_{01} + o(1),
		\end{equation*}
		where $\kappa = A^{-n}$ and $\vec{X}_{01} =  \left(1, (a-1)A \lambda , -\frac{a A^{2 } \lambda^2 }{(m-n)A^m -2n\lambda} \right)$.
	\end{proof}

	In the limit of $\eta \to +\infty$, we have $\lim\limits_{\eta \to +\infty}r(\eta) = B^{m,n,\lambda} > 0$, which gives 
	\begin{equation*}
		0 < B^{m,n,\lambda} = \lim\limits_{\xi \to +\infty} r(\ln \xi) = \lim\limits_{\xi \to +\infty} {\bar{U} \over \bar{\Gamma}^{1\over m} } = \lim\limits_{\xi \to +\infty} \left[ {C_1 \over C_2^{1\over m}} \xi^{{\mu_2 \over m} - \mu_1} + o(\xi^{{\mu_2 \over m} - \mu_1})  \right].
	\end{equation*}
	Necessarily, this requires $\mu_2 = m \mu_1$. At this point, the heuristic calculation for the system \eqref{autoODE} as $\xi \to \infty$ is needed to validate $\mu_2 = m \mu_1$ under the boundary behavior \eqref{eq:BD2}. By substituting \eqref{eq:BD2} to \eqref{autoODE}, we obtain 
	\begin{equation}
		\mu_1 = {1\over m-n}, \quad \mu_2 = {m\over m-n},
	\end{equation}
	under the assumption of 
	\begin{equation*}
		m < n+1, \quad \lambda < {(m-n)(1+m-n) \over m(1-m+n)}.
	\end{equation*}
	The detailed calculation is presented in Appendix~\ref{sec:heu}. \\
	
	\subsection{Linear stability analysis of critical points}~\label{subsec:eigen}
	In this subsection, we will present the linear stability analysis for all four critical points.
	\begin{enumerate}[(i)]
		\item $M_0^{m,n,\lambda} = \left( 0, \ 0, \ A^{m,n,\lambda} \right)$ is an unstable node. The eigenvalues are 
		\begin{equation*}
			\lambda_{01} = 2 >0, \quad \lambda_{02} = 1>0, \quad \lambda_{03} = {(m-n) A^m \over n\lambda} >0,
		\end{equation*}
		with the corresponding eigenvectors
		\begin{align*}
			&\vec{X}_{01} = \left(1,(a-1) A \lambda,-\frac{a A^2 \lambda^2 }{(m-n)A^m -2n\lambda} \right), \\ &\vec{X}_{02} = \left(0,\frac{(m-n)A^m}{\lambda} -n,-A\right), \\ &\vec{X}_{03} = (0,0,1) .
		\end{align*}
		
		\item $M_1^{m,n,\lambda} = \left( 0, \ 1, \ B^{m,n,\lambda} \right)$ is a saddle. The eigenvalues are 
		\begin{equation*}
			\lambda_{11} =-{1-m+n\over m-n} <0 , \quad \lambda_{12} = -1<0, \quad \lambda_{13} = {m-n \over n\lambda} B^m >0,
		\end{equation*}
		with the corresponding eigenvectors
		\begin{align*}
			&\vec{X}_{11} = \left(2-{1\over m-n}, (a-2)B \lambda ,\frac{(a-1)B^2\lambda^2}{(m-n)^2 + \lambda n (1-m+n)} \right), \\ &\vec{X}_{12} = \left(0,\frac{(m-n)B^m }{\lambda}  +n,-B \right), \\ &\vec{X}_{13} = (0,0,1) .
		\end{align*}
		
		\item For $M_2^{m,n,\lambda} = (0,\ 0,\ 0)$, 
		\begin{equation*}
			\lambda_{21} = 2-a-{1\over m\lambda}<0 , \quad \lambda_{22} = 1>0, \quad \lambda_{23} = {a-2 \over n} - {m-n \over nm\lambda}<0 ,
		\end{equation*}
		with the corresponding eigenvectors
		\begin{align*}
			&\vec{X}_{21} = (1,0,0), \\ &\vec{X}_{22} = (0,1,0), \\ &\vec{X}_{23} = (0,0,1) .
		\end{align*}
		
		\item For $M_3^{m,n,\lambda} = (0,\ 1,\ 0)$, 
		\begin{equation*}
			\lambda_{31} = 1-a-{1\over m\lambda}<0 , \quad \lambda_{32} = -1<0, \quad \lambda_{33} = {a-1 \over n} - {m-n \over nm\lambda} <0 ,
		\end{equation*}
		with the corresponding eigenvectors
		\begin{align*}
			&\vec{X}_{31} = (1,0,0), \\ &\vec{X}_{32} = (0,1,0), \\ &\vec{X}_{33} = (0,0,1) .
		\end{align*}
	\end{enumerate}
	
	Proposition~\ref{prop:initial} selects, among the infinitely many heteroclinic orbits connecting $M_0^{m,n,\lambda}$ and $M_1^{m,n,\lambda}$, a particular orbit that yields a smooth localizing solution by fixing its emanating direction to be $\vec{X}_{01}$. From the eigenspace analysis at $M_0^{m,n,\lambda}$, the emanating direction $\vec{X}_{01}$ corresponds not to the principal eigenvector, but to the eigenvector associated with the second largest eigenvalue. This implies that the heteroclinic orbit admitting a smooth localizing solution must lie in the subspace of the unstable manifold of $M_0^{m,n,\lambda}$ spanned by $\vec{X}_{01}$ and $\vec{X}_{03}$. 
	
	Although $M_0^{m,n,\lambda}$ is an unstable node and $M_1^{m,n,\lambda}$ is a saddle, the target orbit is ultimately located at the intersection of the following unstable and stable manifolds.
	\begin{align*}
		W^u(M_0^{m,n,\lambda}) &:= \operatorname{span}(\vec{X}_{01},\vec{X}_{03}), \\ 
		W^s(M_1^{m,n,\lambda}) &:= \operatorname{span}(\vec{X}_{11},\vec{X}_{12}),
	\end{align*}
	\begin{equation*}
		(p,q,r)(\eta)  \in W^u(M_0^{m,n,\lambda}) \cap W^s(M_1^{m,n,\lambda}).
	\end{equation*}
	\begin{figure}[H]
		\centering
		
		\begin{subfigure}{0.40\textwidth}
			\centering
			\includegraphics[width=\textwidth]{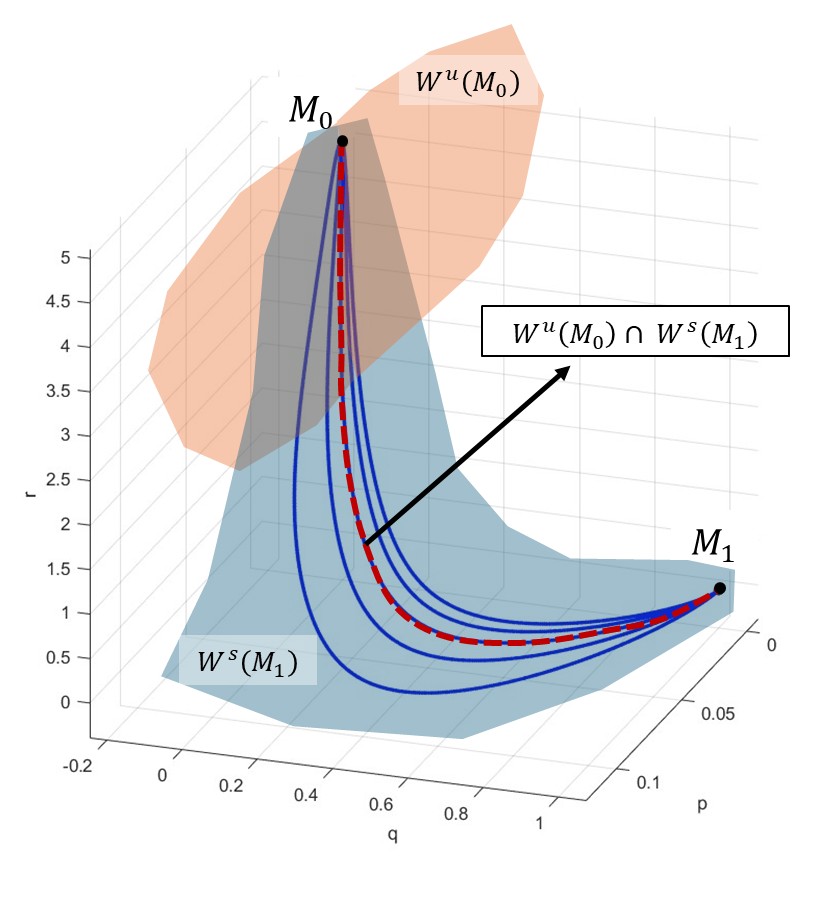}
			\caption{Orbits joining $M_0^{m,n,\lambda}$ and $M_1^{m,n,\lambda}$}
		\end{subfigure}
		\hspace{5mm}
		\begin{subfigure}{0.40\textwidth}
			\centering
			\includegraphics[width=\textwidth]{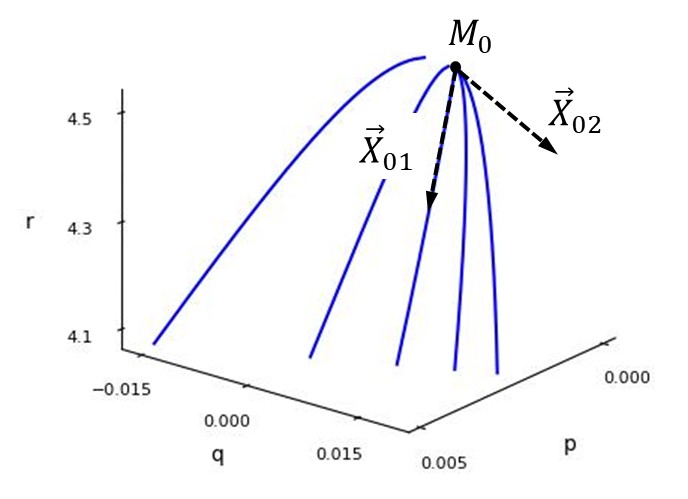}
			\caption{Zoom in near the critical point $M_0^{m,n,\lambda}$}
		\end{subfigure}
		
		\caption{This figure illustrates how the heteroclinic orbit connecting $M_0^{m,n,\lambda}$ and $M_1^{m,n,\lambda}$ is characterized as the intersection of the unstable manifold $W^u(M_0^{m,n,\lambda})$ and the stable manifold $W^s(M_1^{m,n,\lambda})$.}
		\label{fig:manifold}
	\end{figure}
	
	\subsection{Asymptotics of the heteroclinic orbit}~\label{subsec:asymp}
	
	Based on the limit behaviors of $\eta \to \pm \infty$, our goal in this section is to clarify the asymptotics of the heteroclinic orbit joining $M_0^{m,n,\lambda}$ to $M_1^{m,n,\lambda}$. Based on the calculation of Appendix~\ref{sec:heu}, we assume 
	$$ m<n+1, \qquad \lambda < \frac{(m-n)(1+m-n)}{m(1-m+n)}. $$
	If there is a heteroclinic orbit $\varphi(\eta)$ of \eqref{eq:pqr}, it must satisfy
	\begin{equation} \label{unstableM}
		\varphi(\eta) - M_0^{m,n,\lambda} = \kappa e^{2\eta} \vec{X}_{01} + o(e^{2\eta}),
	\end{equation}
	as $\eta \to -\infty$ from Proposition~\ref{prop:initial}. For the limit of $\eta \to +\infty$, we have two subcases. If $m-n \ne {1\over 2}$,
	\begin{equation} \label{stableM1}
		\varphi(\eta) - M_1^{m,n,\lambda} = \kappa_1 e^{-{1-m+n\over m-n} \eta} \vec{X}_{11} + \kappa_2 e^{-\eta} \vec{X}_{12} + \mbox{higher order terms}
	\end{equation} as $\eta \to +\infty$. 
	If $m-n = {1\over 2}$, then we have
	\begin{equation} \label{stableM2}
		\varphi(\eta) - M_1^{m,n,\lambda} = \kappa_1 \eta e^{-\eta} \vec{X}_{11} + \kappa_2 e^{-\eta} \vec{X}_{12} + o(e^{-\eta})
	\end{equation} as $\eta \to +\infty$. 
	Note that for $M_1^{m,n,\lambda}$, only two eigenvalues, $\lambda_{11}$ and $\lambda_{12}$, span the stable eigenspace with the corresponding eigenvectors $\vec{X}_{11}$ and $\vec{X}_{12}$. 
	
	Using \eqref{eq:secondans}, we have the following asymptotics for the $\xi$-dependent functions.
	\begin{proposition}[Asymptotic behavior of $(\bar{U}, \bar{\Gamma}, \bar{\Sigma}, \bar{W})$]\label{prop:asymp}
		Let $(\bar{U}, \bar{\Gamma})$ be a smooth solution of \eqref{nonautoODE} with boundary condition \eqref{eq:BD} and define $\bar{\Sigma}$ and $\bar{W}$ from $\bar{U}$ and $\bar{\Gamma}$. Assume the limit behavior \eqref{eq:BD2} and the parameter domain $0 < n < m < n+1$, $\lambda < \frac{(m-n)(1+m-n)}{m(1-m+n)}$. Let $(p(\eta),q(\eta),r(\eta))$ be the associated variables obtained by the transformation process and $\varphi(\eta)$ be the heteroclinic orbit connecting $M_0^{m,n,\lambda}$ to $M_1^{m,n,\lambda}$. Then, we have
		\begin{enumerate}[(i)]
			\item As $\xi \to 0$, 
			\begin{align*}
				&\bar{U} (\xi) = A -\frac{a\lambda(A^m +2\lambda)}{ 2(m-n)A^m - 4n\lambda }  \ A^{2-n} \xi^2 + o(\xi^2) \\
				&\bar{\Gamma} (\xi) = 1  -\frac{ma\lambda}{2 (m-n)A^m - 4n\lambda }  \ A^{1+m-n} \xi^2  + o(\xi^2) \\
				&\bar{\Sigma} (\xi) = A^n + {1\over 2} a A \lambda \xi^2 + o(\xi^2) \\
				&\bar{W} (\xi) = a A \lambda \xi  + o(\xi^2)
			\end{align*}
			where $a = {2 \over 1+m-n}$ and $A = \left( 1+{2m\over 1+m-n} \lambda \right)^{1\over m}$.
			
			\item As $\xi \to \infty$, if $m-n \ne {1\over 2}$,
			\begin{equation*}
				\bar{U} =  O(\xi^{-{1\over m-n}}), \quad \bar{\Gamma} = O(\xi^{-{m\over m-n}}),
			\end{equation*}
			\begin{equation*}
				\bar{\Sigma} = O(\xi) , \quad \bar{W} = O(1).
			\end{equation*}
			If $m-n = {1\over 2}$, 
			\begin{equation*}
				\bar{U} =  O((\ln \xi)^{2\over 3} \xi^{-{1\over m-n}}), \quad \bar{\Gamma} = O((\ln \xi)^{2m\over 3} \xi^{-{m\over m-n}}),
			\end{equation*}
			\begin{equation*}
				\bar{\Sigma} = O((\ln \xi)^{-{1\over 3}} \xi) , \quad \bar{W} = O((\ln \xi)^{-{1\over 3}}).
			\end{equation*}
		\end{enumerate}
	\end{proposition}
	\begin{proof}
		\begin{enumerate}[(i)]
			\item This was shown in the proof of Proposition~\ref{prop:initial}. 
			
			\item We start from the case of $m-n \ne {1\over 2}$. From the eigenvector expansion \eqref{stableM1}, we have 
			\begin{align*}
				p(\eta) &= C e^{-{1-m+n \over m-n} \eta} + o(e^{-{1-m+n \over m-n} \eta}), \\ q(\eta) &= 1 + o(1), \\ 
				r(\eta) &= B^{m,n,\lambda} + o(1).
			\end{align*}
			Note that only $\vec{X}_{11}$ has a non-zero component for the variable $p$, so we consider the case of $C\ne 0$. By putting these to \eqref{eq:transform}, we have
			\begin{align*}
				U&= (pr^{1+m})^{1\over 1+m-n} = O( e^{-{1-m+n \over (m-n) (1+m-n)} \eta}), \\ 
				\Gamma&= (pr^{n})^{m\over 1+m-n} = O( e^{-{m(1-m+n) \over (m-n) (1+m-n)} \eta}), \\ \Sigma&= (p^{n-m}r^{n})^{1\over 1+m-n} = O( e^{{1-m+n \over 1+m-n} \eta}), \\ W&= \Sigma(q+\lambda pr) = O( e^{{1-m+n \over 1+m-n} \eta}).
			\end{align*}
			Now, using $\xi = e^{\eta}$ and \eqref{eq:secondans}, we have
			\begin{equation*}
				\bar{U} =  O(\xi^{-{1\over m-n}}), \quad \bar{\Gamma} = O(\xi^{-{m\over m-n}}),
			\end{equation*}
			\begin{equation*}
				\bar{\Sigma} = O(\xi) , \quad \bar{W} = O(1).
			\end{equation*}
			For the case $m-n = {1\over 2}$, from \eqref{stableM2}, we have
			\begin{align*}
				p(\eta) &= C \eta e^{- \eta} + o(e^{- \eta}), \\ q(\eta) &= 1 + o(1), \\ 
				r(\eta) &= B^{m,n,\lambda} + o(1).
			\end{align*}
			By putting these to \eqref{eq:transform}, we have
			\begin{align*}
				U&= (pr^{1+m})^{2\over 3} = O( \eta^{2\over 3} e^{-{2 \over  3} \eta}), \\ 
				\Gamma&= (pr^{n})^{2m\over 3} = O( \eta^{2m\over 3} e^{-{2m \over 3} \eta}), \\ \Sigma&= (p^{-{1\over 2}}r^{n})^{2\over 3} = O( \eta^{-{1\over 3}}e^{{1 \over 3} \eta}), \\ W&= \Sigma(q+\lambda pr) = O( \eta^{-{1\over 3}} e^{{1 \over 3} \eta}).
			\end{align*}
			Again, using $\xi = e^{\eta}$ and \eqref{eq:secondans}, we have
			\begin{equation*}
				\bar{U} =  O((\ln \xi)^{2\over 3} \xi^{-{1\over m-n}}), \quad \bar{\Gamma} = O((\ln \xi)^{2m\over 3} \xi^{-{m\over m-n}}),
			\end{equation*}
			\begin{equation*}
				\bar{\Sigma} = O((\ln \xi)^{-{1\over 3}} \xi) , \quad \bar{W} = O((\ln \xi)^{-{1\over 3}}).
			\end{equation*}
		\end{enumerate}
	\end{proof}

%\vfil\eject
%  section fast slow system
%%%%%%%%%%%%%%%%%%%%%%%%%%%%%%%%%%%%%%%%%%%%%%%%%%%%%%%%%%%%%%%%
	
	\section{Approximation to the fast-slow system as \texorpdfstring{$n\to 0$}{n 0} } \label{sec:gspt}
	Since both critical points $M_0^{m,n,\lambda}$ and $M_1^{m,n,\lambda}$ are on the same critical manifold $S^{m,n,\lambda}$ of the function $r$, we can find the heteroclinic orbit by approximating the system \eqref{eq:pqr} to a planar system \eqref{eq:pq0} as $n \to 0$. In this section, we will apply the geometric singular perturbation theory to obtain the heteroclinic orbit joining $M_0^{m,n,\lambda}$ to $M_1^{m,n,\lambda}$ when $d=1$. The goal of this section is to prove the following theorem.
	
	\begin{theorem}\label{thm:main}
		Let $\Lambda$ be a domain of the tuple $(m,n,\lambda) \in \R^3$ defined by
		\begin{align*}
			& n \ge 0, \\ 
			& 0< m < n+1, \\
			& m > n, \\
			& 0<\lambda < \frac{(m-n)(1+m-n)}{m(1-m+n)},\\
			& (m-n)\left(1+\frac{2m}{1+m-n}\lambda \right)-2n\lambda>0.
		\end{align*}
		For each $(m,0,\lambda) \in \Lambda$, there is an $n_0(m,\lambda)$, such that for all $n \in (0,n_0)$, $(\lambda,m,n) \in \Lambda$, the system \eqref{eq:pqr} with $d=1$ admits a heteroclinic orbit joining the equilibrium $M_0^{m,n,\lambda}$ to $M_1^{m,n,\lambda}$ with the following property \begin{equation} \label{eq:asympthm}
			e^{-2\eta} \left( \begin{bmatrix}
				p(\eta) \\ q(\eta) \\ r(\eta)
			\end{bmatrix} - M_0^{m,n,\lambda} \right) \quad \to \quad \kappa \vec{X}_{01} \qquad \mbox{ as } \eta \to -\infty
		\end{equation} for the constant $\kappa$ in \eqref{eq:asymp1}.
	\end{theorem}
	
	\subsection{When \texorpdfstring{$n=0$}{n=0}} 
	 % If we apply $n=0$ to the critical manifold \eqref{def:cm}, we have
     The critical manifold \eqref{def:cm} is then given by
	 \begin{equation}\label{def:cmn0}
		S^{m,0,\lambda} := \left\{ (p,q,r) \mid q + \lambda pr + {1\over \lambda} r^m - {2m \over 1+m} - {1\over \lambda} = 0, \ p \ge 0 \ r\ge 0.  \right\}.
	\end{equation}
	Note that for $f(p,q,r)$ defined as $$f(p,q,r) = q+ \lambda pr + {1\over \lambda} r^m - {2m \over 1+m} - {1\over \lambda},$$ the partial derivative with respect to $r$ is ${\partial f \over \partial r} = \lambda p + {m\over \lambda }r^{m-1}$. It is strictly positive if $r$ is away from 0, therefore we can apply the implicit function theorem in a neighborhood of $M_0^{m,0,\lambda}$ and $M_1^{m,0,\lambda}$. For the $r$ coordinate of $M_1^{m,0,\lambda}$, $B^{m,0,\lambda}$, we will consider the contour line of $S^{m,0,\lambda}$ at $0 < \underline{r} < B^{m,0,\lambda}$. We will choose $\underline{r}$ later, but at this moment, we assume that it is away from $B^{m,0,\lambda}$, that is, $\underline{r} < {1\over 2}B^{m,0,\lambda}$. Now, we consider the following triangular domain 
	\begin{equation}\label{def:dd}
		D := \left\{ (p,q) \in \R^2 \mid p\ge 0, \ q \ge 0, \ q + \lambda p \underline{r} \le {2m \over 1+m} - {1\over \lambda}(\underline{r}^m-1) \right\}
	\end{equation}
	and its image on $S^{m,0,\lambda}$ as $$T^{m,0,\lambda} := \left\{ (p,q,r) \in S^{m,0,\lambda} \mid (p,q) \in D \right\}.$$ For given $m$ and $\lambda$, $T^{m,0,\lambda}$ is always away from the surface $r=0$. However, $D$ and $T^{m,0,\lambda}$ don't have smooth boundary. In order to apply the implicit function theorem, let $\tilde{D}$ be a compact simply connected domain with $C^\infty$ boundary such that $D \subset \tilde{D}$. Define the corresponding image of $\tilde{D}$ on the surface $S^{m,0,\lambda}$ as $G^{m,0,\lambda}$.
	$$G^{m,0,\lambda} := \left\{ (p,q,r) \in S^{m,0,\lambda} \mid (p,q) \in \tilde{D} \right\}.$$
	Still, we can take $G^{m,0,\lambda}$ away from the surface $r=0$. On the bounded surface $G^{m,0,\lambda}$, the implicit function theorem can be applied; there exists a positive function $r:= h(p,q,n = 0)$ such that it is defined on $\tilde{D}$ and $h(p,q,0) \in C^1 (\tilde{D})$. The reduced system for $n=0$ can be written as follows.
	
	\begin{equation}\label{eq:pq0}
		\begin{aligned}
			\dot{p} &= -p \left( q+\lambda ph - \frac{1}{m\lambda} h^m -\frac{2m}{1+m} + {1\over m \lambda} \right),  \\ 
			\dot{q} &= q \ (1-q-\lambda ph) +  \frac{1-m}{1+m} \lambda ph ,\\
			r &= h(p,q,0). 
		\end{aligned} \tag{$P_{pq0}$}
	\end{equation}

	\subsection{The pqr-system in the fast-time scale}
	In this section, we consider a new variable $\tilde{\eta} = {\eta \over \lambda}$, which changes the pqr-system \eqref{eq:pqr} for $d=1$ into fast scale.
	
	\begin{equation}\label{eq:pqrfast}
		\begin{aligned}
			\dot{p} &= -np \left( q+\lambda pr - \frac{1}{m\lambda} r^m -\frac{2(m-n)}{1+m-n} + {1\over m \lambda} \right),  \\ 
			\dot{q} &= n  q \ (1-q-\lambda pr) +  \frac{1-m+n}{1+m-n} \lambda pr ,\\
			\dot{r} &= r\left(q+\lambda pr + \frac{m-n}{m\lambda} r^m -\frac{2(m-n)}{1+m-n} - {m-n\over m\lambda} \right). 
		\end{aligned} 
	\end{equation}
	The above system can be analyzed under the fast-slow dynamics as 
	\begin{equation} \label{eq:fast-slow}
		\begin{aligned}
			\begin{bmatrix} \dot{p} \\ \dot{q} \end{bmatrix} &= n \ g^{m,\lambda}(p,q,r,n), \\
			\dot{r} &= f^{m,\lambda}(p,q,r,n).
		\end{aligned}
	\end{equation} 
	In the limit as $n\to 0$, we have 
	\begin{equation} \label{eq:n0}
		\begin{aligned}
			\begin{bmatrix} \dot{p} \\ \dot{q} \end{bmatrix} &= \vec{0}, \\
			\dot{r} &= f^{m,\lambda}(p,q,r,0).
		\end{aligned}
	\end{equation} 
	To prove Theorem~\ref{thm:main}, we will approximate the flow on $S^{m,n,\lambda}$ to the flow on $S^{m,0,\lambda}$ by the geometric singular perturbation theory when $n$ is small enough. After this approximation, we will investigate the existence of a heteroclinic orbit joining $M_0^{m,n,\lambda}$ to $M_1^{m,n,\lambda}$ with emanating direction $\vec{X}_{01}$.
	
	More precisely, a graph version of Fenichel's first theorem will be used in the proof (Theorem 1 and 2 of \cite{Jones1995}). First, we introduce the notions of normal hyperbolicity and local invariance of a bounded smooth manifold. 
	\begin{definition}[Normal hyperbolicity]
		A manifold $G \subset S^{m,0,\lambda}$ is called normally hyperbolic to \eqref{eq:n0} if $D_r f^{m,\lambda}(p,q,r,0)$ has no eigenvalue with zero real part for all $(p,q,r) \in G$.
	\end{definition}
	\begin{definition}[Local invariance]
		Let $\phi_{\eta}(\cdot)$ be the flow function defined by the vector field of the differential system \eqref{eq:fast-slow}. A manifold $G$ is locally invariant if for all $(p,q,r) \in G$, there exists a time interval $\eta \in [\eta_1,\eta_2]$ such that $\eta_1 < 0 < \eta_2$ and $\phi_{[\eta_1,\eta_2]}(p,q,r) \subset G$.
	\end{definition}
	Fenichel's first theorem asserts the existence of $G^{m,n,\lambda}$ which is a perturbation
	of $G^{m,0,\lambda}$. At the same time, the flow on $G^{m,n,\lambda}$ can be smoothly approximated from the case of $n=0$ by the graph representation $r=h(p,q,n)$. Let $n \in I$ where $I$ is an interval containing 0 and consider an open domain $U\subset \R^3$ which contains $G^{m,0,\lambda}$ and does not intersect with $r=0$. To state the theorem, we assume three hypotheses.
	\begin{enumerate}[(H1)]
		\item $f^{m,\lambda}, g^{m,\lambda} \in C^\infty (U\times I)$.
		
		\item The set $G^{m,0,\lambda}$ is a compact manifold, possibly with boundary, and is normally
		hyperbolic relative to \eqref{eq:n0}.
		
		\item The set $G^{m,0,\lambda}$ is given as the graph of the $C^\infty(\tilde{D})$ function $r=h(p,q,0)$ for $(p,q) \in \tilde{D}$. The set $\tilde{D}$ is a compact, simply connected domain whose boundary is an 1-dimensional $C^\infty$ submanifold.
	\end{enumerate}
	
	Now, we are ready to state the Fenichel's first theorem. 
	\begin{theorem}[Fenichel's first theorem]\label{thm:Fenichel}
		Assume (H1) and (H2). If $n_0>0$ is sufficiently small, for all $n \in (0,n_0)$, there exists a manifold $G^{m,n,\lambda}$ that lies within $O(n)$ of $G^{m,0,\lambda}$ and is diffeomorphic to $G^{m,0,\lambda}$. Moreover it is locally invariant under the flow of \eqref{eq:fast-slow}, $C^k(U\times (0,n_0))$ for any $k < +\infty$.
	\end{theorem}
	\begin{theorem}[Graph version of Theorem~\ref{thm:Fenichel}]\label{thm:graphFenichel}
		Assume (H1), (H2), and (H3). If $n_0>0$ sufficiently small, for all $n \in (0,n_0)$, there exists a function $r = h(p,q,n)$ defined on $(p,q) \in \tilde{D}$ such that the graph 
		\begin{equation}\label{eq:graph}
			G^{m,n,\lambda} = \left\{ (p,q,r)\in \R^3 \mid r= h(p,q,n) \right\}
		\end{equation}
		is locally invariant under \eqref{eq:fast-slow}. Moreover $h \in C^k(\tilde{D}\times (0,n_0))$, for any $k < +\infty$, jointly in $(p,q)$ and $n$.
	\end{theorem}
	
	From its definition, it is clear that $G^{m,0,\lambda} \subset S^{m,0,\lambda}$. When $n > 0$, $S^{m,n,\lambda}$ is a part of the critical manifold of the variable $r$ and $G^{m,n,\lambda}$ is a perturbation of $G^{m,0,\lambda}$. However, we cannot say that $G^{m,n,\lambda} \subset S^{m,n,\lambda}$, and it becomes technically challenging to construct the targeted orbit using Theorem~\ref{thm:Fenichel} without the graph representation \eqref{eq:graph}. As $n$ varies, the critical points $M_0^{m,n,\lambda}$ and $M_1^{m,n,\lambda}$ also change, and the orbit $\phi_{\eta}^{m,n,\lambda}(\cdot)$ varies accordingly. Fenichel's first theorem indicates that, within the given domain $\tilde{D}$, small perturbations in $n$ preserve the local structure, as illustrated in Figure~\ref{fig:gspt}.
	\begin{figure}[H]
		\centering
		\includegraphics[width=0.6\textwidth]{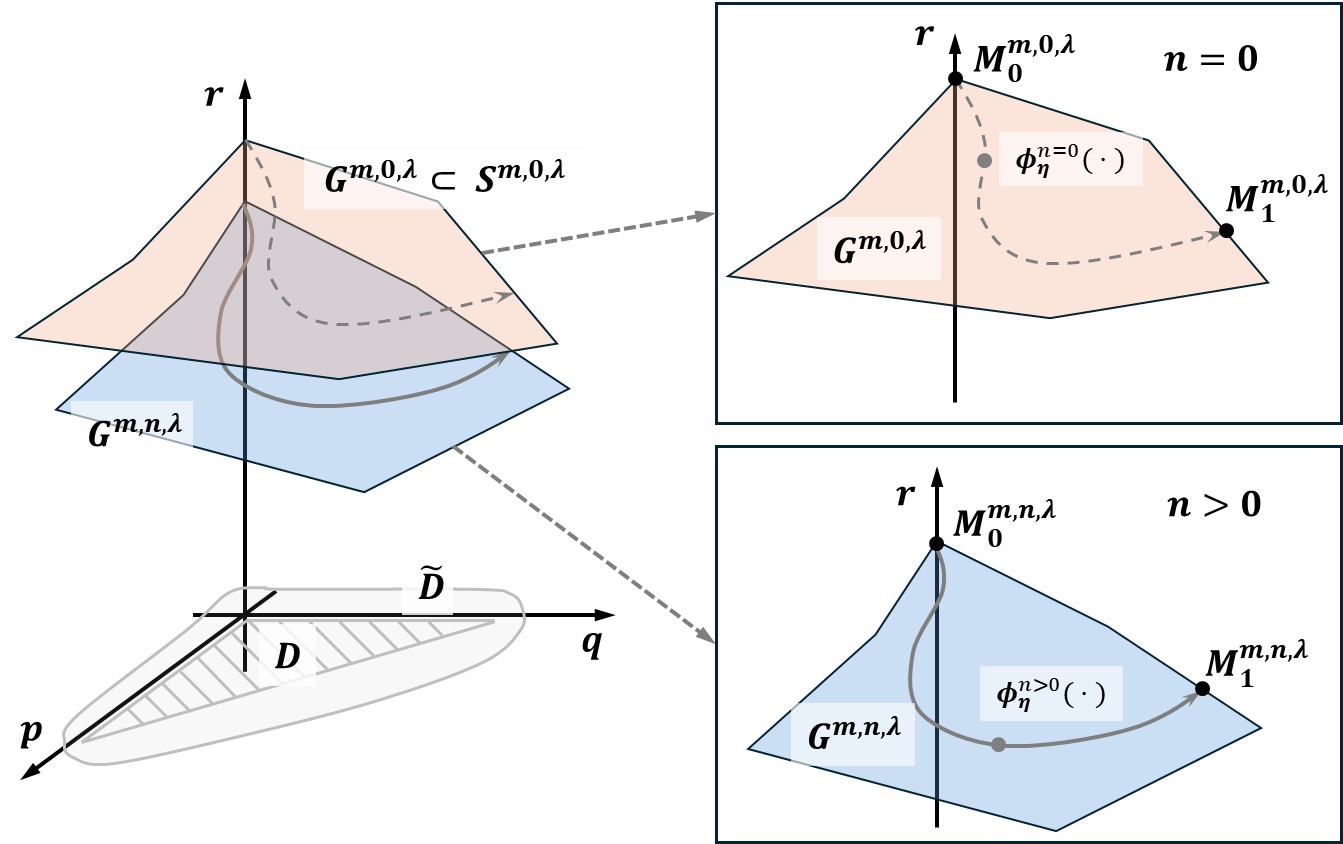}
		\caption{A schematic illustration of Fenichel's first theorem.}
		\label{fig:gspt}
	\end{figure}

	\begin{lemma}\label{lem:norhyp}
		$ G^{m,0,\lambda}$ is a normally hyperbolic manifold with respect to the system \eqref{eq:pq0}. Moreover, $ G^{m,0,\lambda}$ is given as a graph of a function $h(p,q,0) \in C^\infty(\tilde{D})$.
	\end{lemma}
	\begin{proof}
		An application of the implicit function theorem and the graph representation are already shown in the previous section, which is $r=h(p,q,0)$ in \eqref{eq:pq0}. Therefore, it suffices to show that it is normally hyperbolic. By a direct calculation, we have $$D_r f^{0,m,n} = \lambda p + {m-n \over \lambda} r^{m-1}.$$ Since $r>0$ on $ G^{m,0,\lambda}$, $D_r f^{0,m,n}$ is always non-zero, which completes the proof.
	\end{proof}
	
	\subsection{Proof of Theorem~\ref{thm:main}}
	Now, we are ready to prove the main theorem. In the statement of Theorem~\ref{thm:main}, we assume $(m,0,\lambda) \in \Lambda$. Putting it to the definition of $\Lambda$, we have 
	\begin{equation}\label{eq:assump}
		0 < m < 1 , \quad 0 < \lambda < {1+m \over 1-m}.
	\end{equation}
	By definition of $f^{m,\lambda}$ and $g^{m,\lambda}$ in the system \eqref{eq:pqrfast}, these are smooth functions on $U\times I$. By Lemma~\ref{lem:norhyp}, the Fenichel's first theorem \eqref{thm:graphFenichel} can be applied. Thus, for any given $m$ and $\lambda$ with $(m,0,\lambda) \in \Lambda$, there exists a small $n_0>0$ such that for every $n \in (0,n_0)$, there exists a perturbed manifold $G^{m,n,\lambda}$ and is given by a graph $(p,q,h(p,q,n))$, which is a locally invariant manifold with respect to the system \eqref{eq:pqrfast}. The graph function $h(p,q,n) \in C^k (\tilde{D}\times (0,n_0))$, for any $k < +\infty$. By Theorem~\ref{thm:Fenichel}, $G^{m,n,\lambda}$ is $O(n)$ order perturbation of $G^{m,0,\lambda}$. Thus, if it is needed, we take $n_0$ smaller such that $G^{m,n,\lambda}$ does not intersect with $r=0$ and $B^{m,n,\lambda} > {1\over 2} B^{m,0,\lambda} > \underline{r}$. If necessary, we decrease $n_0>0$ further so that
		\[
		n_0<m, \qquad \lambda<\frac{(m-n_0)(1+m-n_0)}{m(1-m+n_0)}, \qquad (m-n_0)\left(1+\frac{2m}{1+m-n_0}\lambda\right)-2n_0\lambda>0.
		\]
		These inequalities then hold uniformly for every $n\in(0,n_0)$. The first inequality ensures that $0<n<m$, while the second guarantees the positivity of the $r$-coordinate of $M_1^{m,n,\lambda}$. Thus, the two critical points $M_0^{m,n,\lambda}$ and $M_1^{m,n,\lambda}$ are well defined and lie in the relevant region of the positive phase space. The third inequality guarantees that
		\[
		(m-n)(A^{m,n,\lambda})^m-2n\lambda>0
		\]
		for every $n\in(0,n_0)$, which is the condition ensuring that the distinguished eigenvector $\vec{X}_{01}$ points into the positively invariant region introduced below.

	By putting $r =  h(p,q,n)$ to the system \eqref{eq:pqrfast}, the reduced system of $(p,q)$ is obtained as
	\begin{equation}\label{eq:pqreduced}
		\begin{aligned}
			\dot{p} &= -p \left( q+\lambda ph - \frac{1}{m\lambda} h^m -\frac{2(m-n)}{1+m-n} + {1\over m \lambda} \right),  \\ 
			\dot{q} &= q \ (1-q-\lambda ph) +  \frac{1-m+n}{1+m-n} \lambda ph ,\\
			r &= h(p,q,n). 
		\end{aligned} \tag{$P_{pq}$}
	\end{equation}

	\begin{lemma}\label{lem:critical}
		The system \eqref{eq:pqreduced} on the manifold $ G^{m,n,\lambda}$ allows only two critical points, $M_0^{m,n,\lambda}$ and $M_1^{m,n,\lambda}$.
	\end{lemma}
	\begin{proof}
		Since $ G^{m,n,\lambda}$ does not intersect with $r=0$, only $M_0^{m,n,\lambda}$ and $M_1^{m,n,\lambda}$ can be critical points of \eqref{eq:pqreduced}. By putting $(0,0)$ and $(0,1)$, we can easily verify that they are critical points of the system \eqref{eq:pqreduced}. Since $\dot{r} = \pdv{h}{p} \dot{p} + \pdv{h}{q} \dot{q}$, the two points $(0,0,h(0,0,n))$ and $(0,1,h(0,1,n))$ should be the critical points of the original system \eqref{eq:pqrfast}. Necessarily, they are $M_0^{m,n,\lambda} = (0,0,h(0,0,n))$ and $M_1^{m,n,\lambda} = (0,1,h(0,1,n))$.
	\end{proof}
	
	Now, we recall the definition of the triangular domain $D$ in \eqref{def:dd}. Here, we define its image on $G^{m,n,\lambda}$ as follows.
	\begin{equation}
		T^{m,n,\lambda} := \left\{ (p,q,r) \in G^{m,n,\lambda} \mid (p,q) \in D \right\}.
	\end{equation}
	We will show that the above triangular domain is a positively invariant region of the system \eqref{eq:pqreduced} as in Figure~\ref{fig:invariant}. 
	\begin{figure}[H]
		\centering
		\includegraphics[width=0.35\textwidth]{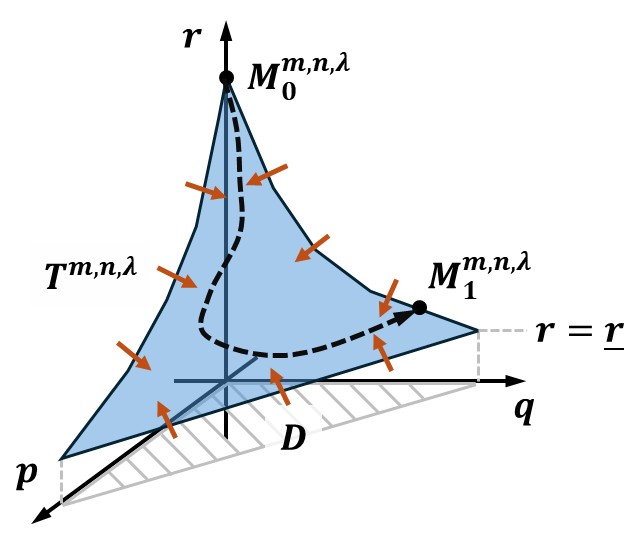}
		\caption{A schematic illustration of a positively invariant region $T^{m,n,\lambda}$.}
		\label{fig:invariant}
	\end{figure}
	\begin{lemma}
		If $0<n<m$, there exists a small enough but positive $\underline{r}$ such that $T^{m,n,\lambda}$ is a positively invariant region of the system \eqref{eq:pqreduced}. 
	\end{lemma}
	\begin{proof}
		It is enough to show that on each boundary, $$(\dot{p},\dot{q})\cdot \mathbf{n} \ge 0$$ when $\mathbf{n}$ is an inward normal vector. On the line $p=0$, $\mathbf{n} = (1,0)$ and $\dot{p} = 0$. Thus, $$(\dot{p},\dot{q})\cdot \mathbf{n} \mid_{\left\{ p=0 \right\}} = 0.$$ On the line $q=0$, $\mathbf{n} = (0,1)$ and we have
		$$(\dot{p},\dot{q})\cdot \mathbf{n} \mid_{\left\{ q=0 \right\}} = \dot{q}(p,0) = \frac{1-m+n}{1+m-n} \lambda ph >0$$ since $m < n+1$ and $h>0$. On the line $L : q + \lambda p \underline{r} = {2m \over 1+m} - {1\over \lambda}(\underline{r}^m-1)$, the inward normal vector is $\mathbf{n} = (-\lambda \underline{r},-1)$ and we have 
		\begin{align*}
			(\dot{p}, \dot{q}) \cdot \mathbf{n} |_{\left\{L\right\}} &= -\lambda \underline{r} \dot{p}-\dot{q} \\ & = \frac{1-m+n}{1+m-n} \lambda p ( \underline{r}- h) - (\lambda \underline{r} p+q) (1-q-\lambda ph)  - \frac{ \underline{r}p}{m} (h^m-1) \\ &= \left( {2m\over 1+m} + {1\over \lambda} -\frac{1-m+n}{1+m-n}\right) \lambda p h + \left( {2m\over 1+m} + {1\over \lambda} \right) \left( -{1-m \over 1+m} + {1\over \lambda} \right)  + O(\underline{r}^m) .
		\end{align*}
		Since we are on the parameter domain \eqref{eq:assump}, the above quantities can be estimated as 
		$$ {1\over \lambda} > {1-m \over 1+m}, \quad {2m \over 1+m} + {1\over \lambda} > 1. $$ From the fact that $n < m$, we have $\frac{1-m+n}{1+m-n} < 1$. By combining everything together, 
		\begin{equation*}
			(\dot{p}, \dot{q}) \cdot \mathbf{n} |_{\left\{L \right\}} = \delta + O(\underline{r}^m) 
		\end{equation*}
		for some positive $\delta>0$. By taking $\underline{r}>0$ small enough, we can conclude that $(\dot{p}, \dot{q}) \cdot \mathbf{n} |_{\left\{L \right\}} \ge 0$ for all $(p,q) \in L$.
	\end{proof}
	
	Note that $M_0^{m,n,\lambda}$ is placed on the boundary of $T^{m,n,\lambda}$ and the emanating direction $\vec{X}_{01}$ is pointing inward to the triangular domain. The orbit emanating in $\vec{X}_{01}$ from $M_0^{m,n,\lambda}$ is continued to the interior of $T^{m,n,\lambda}$ by the stable manifold theorem. Since $M_0^{m,n,\lambda}$ and $M_1^{m,n,\lambda}$ are the only critical points by Lemma~\ref{lem:critical} and they are placed on the boundary, there is neither a periodic orbit nor a homoclinic orbit inside $T^{m,n,\lambda}$. Thus, by Poincar\'e–Bendixson theorem, there is a heteroclinic orbit from $M_0^{m,n,\lambda}$ to $M_1^{m,n,\lambda}$ on the manifold $G^{m,n,\lambda}$. This heteroclinic orbit is the desired orbit of the system \eqref{eq:pqr} and \eqref{eq:asympthm} follows from Proposition~\ref{prop:initial}. This completes the proof of Theorem~\ref{thm:main}.
	
	\vspace{5mm}
	
	\begin{proof}[Proof of Theorem~\ref{thm:exist}]
		Since we have $$0<m<1, \quad 0<\lambda < {1+m\over 1-m},$$ for the parameter tuple $(m,n,\lambda)$, we have $(m,0,\lambda) \in \Lambda$. Then, we can apply Theorem~\ref{thm:main} which allows us to find $n_0(m,\lambda)$ such that there exists a heteroclinic orbit from $M_0^{m,n,\lambda}$ to $M_1^{m,n,\lambda}$ to the system \eqref{eq:pqr} for all $n\in (0,n_0)$. In light of Theorem~\ref{thm:main}, the heteroclinic orbit automatically satisfies \eqref{eq:assumpn0}. Denote the heteroclinic orbit of \eqref{eq:pqr} as $\varphi^{m,n,\lambda}(\eta) = (p(\eta), q(\eta), r(\eta))$.
		
		Now, we can apply Proposition~\ref{prop:asymp} to the orbit $\varphi^{m,n,\lambda}$, which gives the dynamics of $\bar{U}$ and $\bar{\Gamma}$ as $\xi \to 0$ or $\xi \to +\infty$. It can be immediately verified that the recovered $(\bar{U}, \bar{\Gamma})$ satisfies the boundary condition \eqref{eq:BD} and \eqref{eq:BD2} with the polynomial exponents $$\mu_1 = {1\over m-n}, \quad \mu_2 = {m\over m-n}.$$ Since $m<1$, we have $\mu_1 > a^{m,n}$ and $\mu_2 > b^{m,n}$. 
		
		Using Proposition~\ref{prop:growth}, the recovered solution $(u,\gamma)$ from \eqref{eq:firstans} is a localizing solution and the function evaluation at $x=0$ is derived from the boundary condition \eqref{eq:BD} and \eqref{eq:lU0} directly.
	\end{proof}

%\vfil\eject
%  section numerical simulations
%%%%%%%%%%%%%%%%%%%%%%%%%%%%%%%%%%%%%%%%%%%%%%%%%%%%%%%%%%%%%%%%

	\section{Asymptotic behavior of the heteroclinic orbit and numerical simulations} \label{sec:numeric}
	
	From Theorem \ref{thm:exist}, we establish the existence of a smooth localizing solution for the given parameters $m$, $n$, and $\lambda$, and the corresponding heteroclinic orbit is provided by Theorem \ref{thm:main}. Furthermore, Proposition \ref{prop:asymp} characterizes the asymptotic behavior that the identified heteroclinic orbit must satisfy with respect to the self-similar variable $\xi = e^{\lambda t} \rho$. By tracing back the nonlinear transformation \eqref{eq:recovered}, the original solution is given as follows.
	
	\begin{equation} \label{eq:selfsimilarprofile}
		\begin{aligned}
			u(\rho,t) &= e^{a \lambda t} \bar{U}(e^{\lambda t} \rho) = \rho^{-a} U( \ln \rho + \lambda t ),\\
			\gamma(\rho,t) &= e^{b \lambda t} \bar{\Gamma}(e^{\lambda t} \rho) = \rho^{-b} \Gamma( \ln \rho + \lambda t ).
		\end{aligned}
	\end{equation}
	
	In this section, we aim to describe the asymptotic behavior of the original solution $(u, \gamma)$ associated with the obtained heteroclinic orbit and to verify it numerically. In the analysis below, we exclude the special case $m - n = \tfrac{1}{2}$, where the eigenvalues have multiplicity and a logarithmic correction arises. 
	
	Due to the symmetric structure of $\xi = e^{\lambda t} \rho$, the solutions considered in this paper exhibit a localization phenomenon in which mass concentrates at $\rho = 0$. The solution has a hill-shaped profile, displaying polynomial decay with respect to the spatial variable, while the localization becomes increasingly pronounced as time evolves. The asymptotic orders for each region are given as follows.
	
	%\color{blue}
	Here, we recall from Proposition~\ref{prop:asymp} the asymptotic behavior of $\bar{U}(\xi)$ and $\bar{\Gamma}(\xi)$.
	
	\begin{enumerate}[(i)]
		\item As $\xi\to0$,
		\begin{align*}
			\bar{U}(\xi)&=A-\frac{a\lambda(A^m+2\lambda)}{2(m-n)A^m-4n\lambda}A^{2-n}\xi^2+o(\xi^2),\\
			\bar{\Gamma}(\xi)&=1-\frac{ma\lambda}{2(m-n)A^m-4n\lambda}A^{1+m-n}\xi^2+o(\xi^2),
		\end{align*}
		where
		\[
		a=\frac{2}{1+m-n}, \qquad A=\left(1+\frac{2m}{1+m-n}\lambda\right)^{1/m}.
		\]
		
		\item As $\xi\to\infty$, if $m-n\neq\frac12$,
		\[
		\bar{U}(\xi)=O\left(\xi^{-\frac{1}{m-n}}\right), \qquad \bar{\Gamma}(\xi)=O\left(\xi^{-\frac{m}{m-n}}\right).
		\]
	\end{enumerate}
	
	Substituting these estimates into the self-similar representation \eqref{eq:selfsimilarprofile}, we obtain the following behavior of the original solution.
	
	\begin{itemize}
		\item At $\rho=0$, the solution attains its maximum, which grows exponentially in time. More precisely,
		\[
		u(0,t)=Ae^{\frac{2}{1+m-n}\lambda t}, \qquad \gamma(0,t)=e^{\frac{2m}{1+m-n}\lambda t}.
		\]
		
		\item Let $\rho\neq0$ be fixed and assume that $m-n\neq\frac12$. As $t\to\infty$, the solution decays exponentially:
		\begin{align*}
			u(\rho,t)&=\rho^{-\frac{1}{m-n}}O\left(e^{-\frac{1-m+n}{(1+m-n)(m-n)}\lambda t}\right),\\
			\gamma(\rho,t)&=\rho^{-\frac{m}{m-n}}O\left(e^{-\frac{m(1-m+n)}{(1+m-n)(m-n)}\lambda t}\right).
		\end{align*}
		Since $n<m<n+1$, both exponential rates are negative.
	\end{itemize}
	\color{black}
	
	To numerically verify the derived asymptotic behavior, we perform simulations for the case $n = 0.3$ and $m = 0.9$. First, the growth of the maximum value of the solution at the peak exhibits exponential behavior, as shown in Figure~\ref{fig:maximum}. The measured growth rate agrees with the theoretical value, $$exponential \ growth \ order = {2 \over 1 + m - n}\lambda \approx 3.2083.$$
	\begin{figure}[H]
		\centering
		\includegraphics[width=0.8\textwidth]{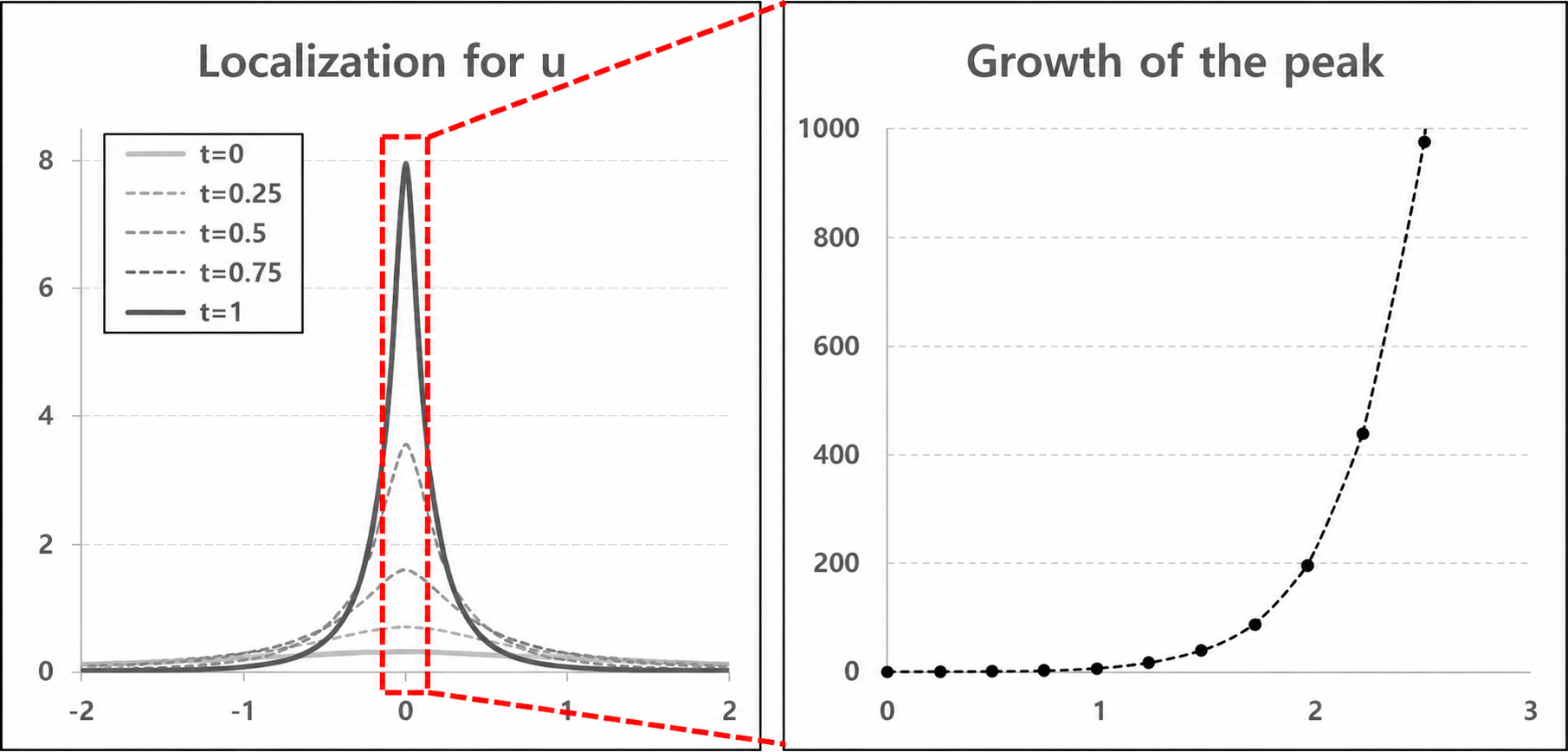}
		\caption{Exponential growth of the peak when $n=0.3$, $m=0.9$.}
		\label{fig:maximum}
	\end{figure}
	Next, we examine the decay rate of the solution $u(\rho,t)$ at $\rho = 2$ as time increases, shown in Figure~\ref{fig:decay}. 
	\begin{figure}[H]
		\centering
		\includegraphics[width=0.8\textwidth]{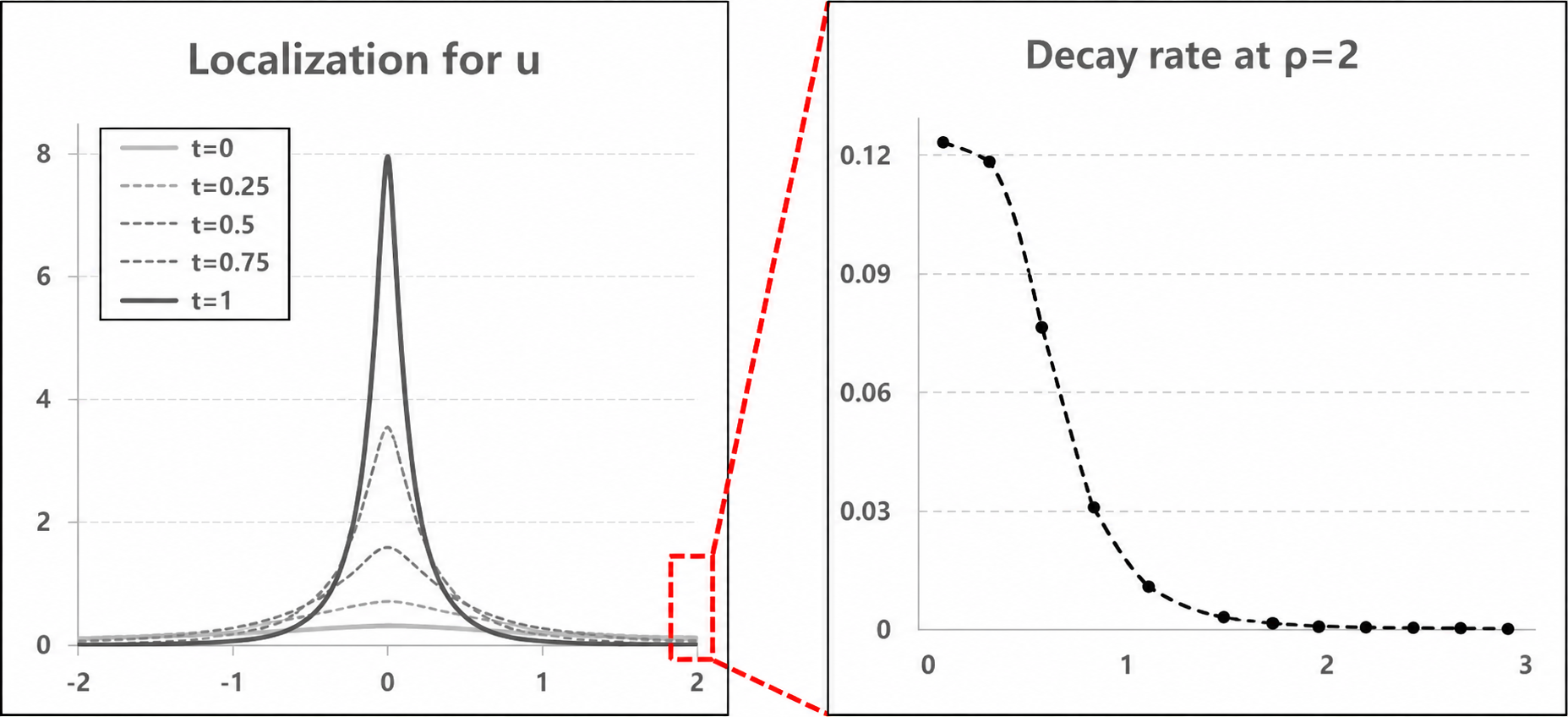}
		\caption{Exponential decay at $\rho =2$ when $n=0.3$, $m=0.9$.}
		\label{fig:decay}
	\end{figure}
	For small times, the solution does not closely follow the exponential decay, and due to the influence of the conservation law, it tends to exhibit a slower decay than the predicted exponential rate. However, as time becomes large, we observe in Figure~\ref{fig:decay2} that the decay rate approaches the theoretical value 
	\[
	exponential \ decay \ order = \left( -\frac{1 - m + n}{(1 + m - n)(m - n)} \right)\lambda \approx -1.0694.
	\]
	\begin{figure}[H]
		\centering
		\includegraphics[width=0.4\textwidth]{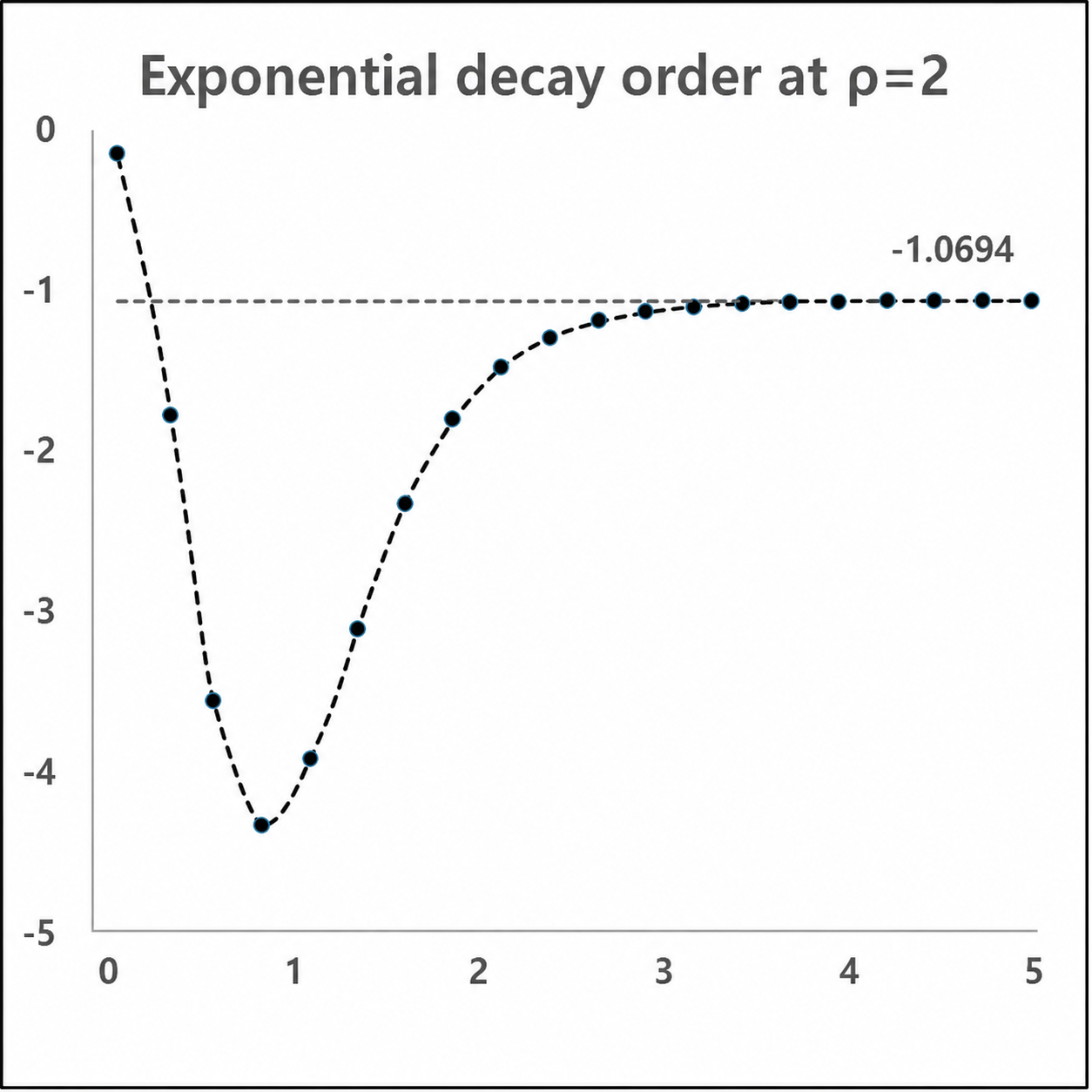}
		\caption{A direct computation of the exponential decay rate. The gray dashed line indicates the theoretical value.}
		\label{fig:decay2}
	\end{figure}

%%%%%%%%%%%%%%%%	
%	\newpage
	
\appendix

%%%%%%%
%   Numerical simulations	

\section{Numerical simulations of the diffusion-relaxation model}\label{appA}

We provide in this appendix numerical simulations for the initial value problem associated with the diffusion--relaxation model \eqref{pmi} on the one-dimensional domain $\Omega=(0,1)$, subject to Neumann boundary conditions. The numerical solutions of \eqref{eq:radial} are computed using a fully implicit finite volume scheme combined with Newton iteration. Let $t_k=k\Delta t$ denote the discrete time levels. Applying the backward Euler method to the relaxation equation gives
\[
\frac{\gamma_i^{k+1}-\gamma_i^k}{\Delta t}=-\gamma_i^{k+1}+(u_i^{k+1})^m,
\]
and hence
\[
\gamma_i^{k+1}=\frac{\gamma_i^k+\Delta t\,(u_i^{k+1})^m}{1+\Delta t}.
\]
Thus, $\gamma_i^{k+1}$ can be eliminated algebraically from the discrete system as a function of $\gamma_i^k$ and $u_i^{k+1}$. Substitution of this relation into the diffusion equation reduces the coupled update to a nonlinear implicit equation for $u^{k+1}$ alone, which is the following relation. 
$$
\frac{u_i^{k+1}-u_i^k}{\Delta t} = \Delta \left( \frac{u_i^{k+1}}{\gamma_i^{k+1}(\gamma_i^k, u_i^{k+1})} \right) 
$$
The diffusion operator is discretized in space by the finite volume method, with the Neumann boundary conditions imposed through zero numerical fluxes at the boundary. The resulting nonlinear algebraic system is solved by Newton iteration, and $\gamma^{k+1}$ is subsequently recovered from the formula above. The initial data are chosen to be smooth, strictly positive, and compatible with the imposed symmetry and boundary conditions.

%\tcr{\bf Please check the statement above how the runs were done, numerical scheme etc} % Comment by Hoyoun : Checked, and it is correct

Below we show numerical runs for dimension $d=1$ for initial data in a trigonometric form. Depending on the given parameter regime, solutions exhibit two distinct behaviors: In the range $n > m > 0$ solutions equilibrate approaching a constant state. By contrast  when $m > n > 0$ a coherent localized structure emerges. These phenomena were observed numerically, and are illustrated in Figure~\ref{fig:EquilvsLocal}. 
\begin{figure}[H]
	\centering
	
	\begin{subfigure}{0.8\textwidth}
		\centering
		\includegraphics[width=\textwidth]{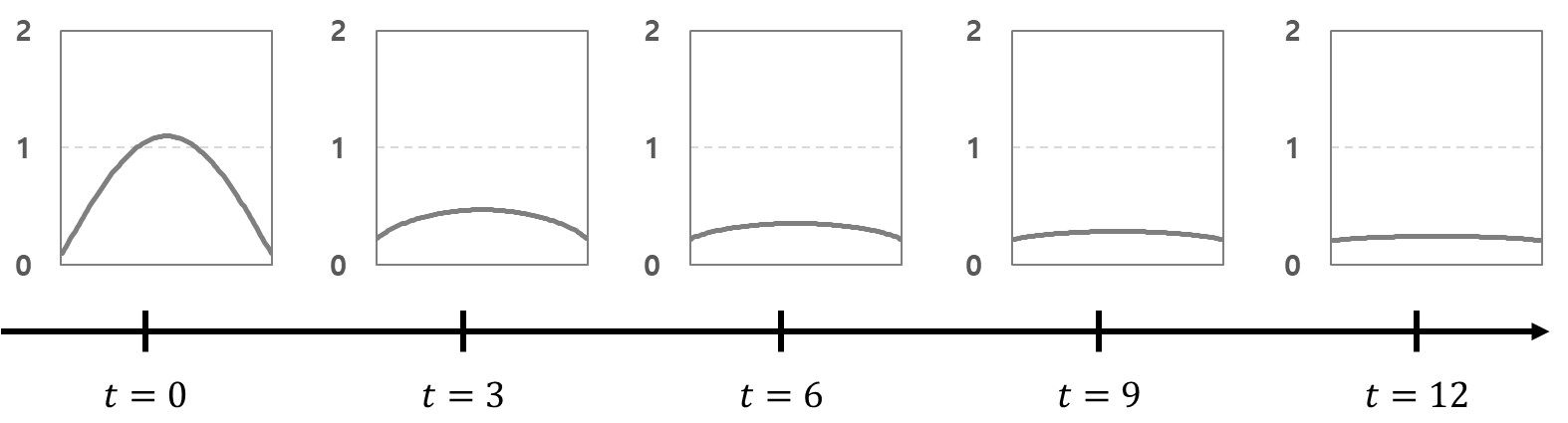}
		\caption{Equilibration, $n=2.5$, $m=2.0$.}
	\end{subfigure}
	
	\vspace{0.5cm}
	
	\begin{subfigure}{0.8\textwidth}
		\centering
		\includegraphics[width=\textwidth]{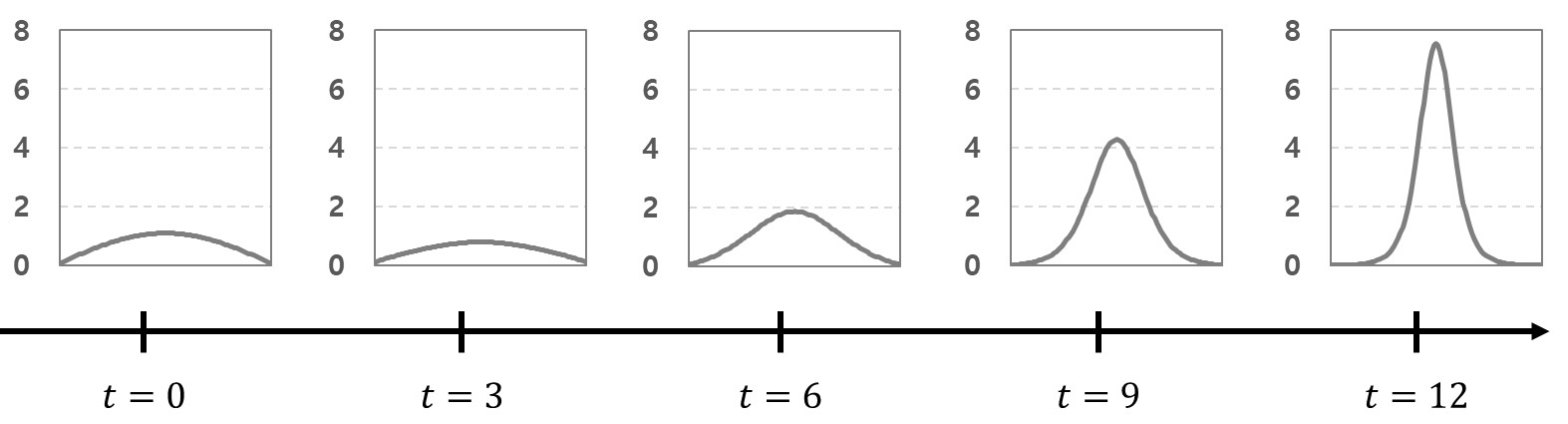}
		\caption{Localization, $n=2.0$, $m=2.5$.}
	\end{subfigure}
	
	\caption{The figure shows simulations of the one-dimensional diffusion--relaxation model for two parameter regimes. A smooth initial condition of trigonometric form is assumed.}
	\label{fig:EquilvsLocal}
\end{figure}
That a difference of behavior is expected can be conjectured  by a simple argument. As time increases the relaxation equation drives the behavior near the equilibrium curve $\gamma = u^m$. The effective response is captured by the diffusion equation $\del_t u = \Delta u^{n-m}$ which is stable for $n> m$ but unstable for $n < m$. The formal limiting equation in the regime $n < m$ is ill-posed at the linearized level and one would expect the development of wild oscillations. Oscillations are not observed numerically and the reason appears to be that the relaxation with the nonlinearity provides a subtle regularizing mechanism. 
This regularizing effect is captured asymptotically by the Chapman--Enskog expansion in the high-relaxation limit. It shows that the next order of the expansion offers a stabilizing mechanism in the unstable range; see Section~\ref{sec:CE}.

%\vfil\eject
%%%%%%%
%   Heuristic Calculation
	
	\section{Heuristic derivation of the far-field asymptotics}\label{sec:heu}
	
	In this appendix, we derive the expected far-field behavior of a localizing profile. We assume that the asymptotic expansions in \eqref{eq:BD2} can be differentiated to the order required below. Substituting
	\[
	\bar{U}(\xi)=C_1\xi^{-\mu_1}+o(\xi^{-\mu_1}), \qquad \bar{\Gamma}(\xi)=C_2\xi^{-\mu_2}+o(\xi^{-\mu_2})
	\]
	into \eqref{nonautoODE} gives
	\begin{align}
		&\lambda(a-\mu_1)C_1\xi^{-\mu_1}+o(\xi^{-\mu_1})=\frac{C_1^n}{C_2}(\mu_2-n\mu_1)(\mu_2-n\mu_1+d-2)\xi^{-n\mu_1+\mu_2-2}+o(\xi^{-n\mu_1+\mu_2-2}), \label{eq:heu-first}\\
		&\bigl(1+\lambda(b-\mu_2)\bigr)C_2\xi^{-\mu_2}+o(\xi^{-\mu_2})=C_1^m\xi^{-m\mu_1}+o(\xi^{-m\mu_1}). \label{eq:heu-second}
	\end{align}
	For the terms in \eqref{eq:heu-second} to have the same asymptotic order, we require
	\begin{equation}\label{eq:heu-relaxation}
		\mu_2=m\mu_1.
	\end{equation}
	The leading-order coefficients then satisfy
	\begin{equation}\label{eq:heu-coefficient}
		\frac{C_1^m}{C_2}=1+\lambda(b-\mu_2).
	\end{equation}
	
	We next consider the first profile equation. Since localization requires $\mu_1>a$ and $a=2/(1+m-n)$, we have
	\[
	\bigl(-n\mu_1+\mu_2-2\bigr)-(-\mu_1)=(1+m-n)\mu_1-2>0.
	\]
	Thus, the term on the right-hand side of \eqref{eq:heu-first} decays more slowly than the term on the left-hand side. Its leading-order coefficient must therefore vanish:
	\[
	(\mu_2-n\mu_1)(\mu_2-n\mu_1+d-2)=0.
	\]
	Using \eqref{eq:heu-relaxation}, this becomes
	\[
	(m-n)\mu_1\bigl((m-n)\mu_1+d-2\bigr)=0.
	\]
	Since $m>n$ and $\mu_1>0$, the first factor is nonzero. Hence,
	\[
	\mu_1=\frac{2-d}{m-n}.
	\]
	In particular, for $d=1$,
	\begin{equation}\label{eq:heu-exponents}
		\mu_1=\frac{1}{m-n}, \qquad \mu_2=\frac{m}{m-n}.
	\end{equation}
	This calculation also explains why the present construction is restricted to $d=1$: for $d\geq2$, the above balance does not produce a positive decay exponent $\mu_1$.
	
	For $d=1$, the expected asymptotic behavior is therefore
	\[
	\bar{U}(\xi)=C_1\xi^{-\frac{1}{m-n}}+o\left(\xi^{-\frac{1}{m-n}}\right), \qquad \bar{\Gamma}(\xi)=C_2\xi^{-\frac{m}{m-n}}+o\left(\xi^{-\frac{m}{m-n}}\right).
	\]
	Since $\bar{\Sigma}=\bar{U}^n/\bar{\Gamma}$ and $\bar{W}=\bar{\Sigma}'$, we also obtain
	\[
	\bar{\Sigma}(\xi)=\frac{C_1^n}{C_2}\xi+o(\xi), \qquad \bar{W}(\xi)=\frac{C_1^n}{C_2}+o(1).
	\]
	The localization conditions $\mu_1>a$ and $\mu_2>b$ are both equivalent to
	\[
	m<n+1.
	\]
	Moreover, substituting \eqref{eq:heu-exponents} into \eqref{eq:heu-coefficient} gives
	\begin{equation}\label{eq:heu-coefficient-final}
		\frac{C_1^m}{C_2}=1-\frac{m\lambda(1-m+n)}{(m-n)(1+m-n)}.
	\end{equation}
	Because $C_1$ and $C_2$ are positive and $m<n+1$, we must have
	\begin{equation}\label{eq:heu-lambda}
		0<\lambda<\frac{(m-n)(1+m-n)}{m(1-m+n)}.
	\end{equation}
	
	We finally express the far-field behavior in terms of the variables $(p,q,r)$. From their definitions,
	\begin{align*}
		p(\ln\xi)&=\frac{\xi^2\bar{\Gamma}^{1/m}}{\bar{\Sigma}}=\frac{C_2^{1+1/m}}{C_1^n}\xi^{1-\frac{1}{m-n}}+o\left(\xi^{1-\frac{1}{m-n}}\right),\\
		q(\ln\xi)&=\frac{\xi\bar{W}-\lambda\xi^2\bar{U}}{\bar{\Sigma}}=1-\frac{\lambda C_2}{C_1^{n-1}}\xi^{1-\frac{1}{m-n}}+o\left(\xi^{1-\frac{1}{m-n}}\right),\\
		r(\ln\xi)&=\frac{\bar{U}}{\bar{\Gamma}^{1/m}}=\frac{C_1}{C_2^{1/m}}+o(1).
	\end{align*}
	Since $m-n<1$, we have $1-1/(m-n)<0$. Therefore,
	\[
	\lim_{\xi\to\infty}p(\ln\xi)=0, \qquad \lim_{\xi\to\infty}q(\ln\xi)=1,
	\]
	and \eqref{eq:heu-coefficient-final} gives
	\[
	\lim_{\xi\to\infty}r(\ln\xi)=\frac{C_1}{C_2^{1/m}}=\left(1-\frac{m\lambda(1-m+n)}{(m-n)(1+m-n)}\right)^{1/m}=B^{m,n,\lambda}.
	\]
	Equivalently,
	\begin{equation}
		\begin{bmatrix}p(\eta)\\q(\eta)\\r(\eta)\end{bmatrix}\longrightarrow M_1^{m,n,\lambda} \qquad\text{as }\eta\to+\infty.
	\end{equation}
	Thus, the far-field conditions \eqref{eq:BD2} correspond to convergence toward $M_1^{m,n,\lambda}$ in the autonomous formulation. The expected heteroclinic orbit therefore connects $M_0^{m,n,\lambda}$ to $M_1^{m,n,\lambda}$ under the parameter restrictions
	\[
	0<n<m<n+1, \qquad 0<\lambda<\frac{(m-n)(1+m-n)}{m(1-m+n)}.
	\]
		
        \bibliographystyle{plain}
		\bibliography{localref}
		
\end{document}